\documentclass[11pt]{amsart}

\usepackage{fullpage}
\usepackage{amsmath}
\usepackage{amsthm}
\usepackage{amsfonts}
\usepackage{amssymb}

\newcommand{\mb}{\mathbf}
\newcommand{\mc}{\mathcal}
\renewcommand{\Re}{\mathrm{Re}}

\newtheorem{lemma}{Lemma}[section]
\newtheorem{proposition}[lemma]{Proposition}
\newtheorem{theorem}[lemma]{Theorem}

\theoremstyle{remark}

\theoremstyle{definition}
\newtheorem{definition}[lemma]{Definition}

\title{On stable self--similar blow up for equivariant wave maps}

\author{Roland Donninger}
\address{University of Chicago, Department of Mathematics,
5734 South University Avenue, Chicago, IL 60637, U.S.A.}
\email{donninger@uchicago.edu}
\thanks{The author 
is an Erwin Schr\"odinger Fellow of the 
FWF (Austrian Science Fund) Project No. J2843 and he wants to thank Wilhelm Schlag for many interesting discussions.}

\begin{document}

\begin{abstract}
We consider co--rotational wave maps from ($3+1$) Minkowski space into the three--sphere. This is an energy supercritical model which is known to exhibit finite time blow up via self--similar solutions. The ground state self--similar solution $f_0$ is known in closed form and based on numerics, it is supposed to describe the generic blow up behavior of the system. We prove that the blow up via $f_0$ is stable under the assumption that $f_0$ does not have unstable modes.
This condition is equivalent to a spectral assumption for a linear second order ordinary differential operator.
In other words, we reduce the problem of stable blow up to a linear ODE spectral problem.
Although we are unable, at the moment, to verify the mode stability of $f_0$ rigorously, it is known that possible unstable eigenvalues are confined to a certain compact region in the complex plane.
As a consequence, highly reliable numerical techniques can be applied and all available results strongly suggest the nonexistence of unstable modes, i.e., the assumed mode stability of $f_0$.
\end{abstract}

\maketitle

\section{Introduction}
Wave maps are (formally) critical points of the action functional
$$ S(u)=\int_M \mathrm{tr}_g (u^*h) $$
for a map $u: M \to N$ where $(M,g)$ and $(N,h)$ are Lorentzian and Riemannian manifolds, respectively.
Here, $\mathrm{tr}_g (u^*h)$ is the trace (with respect to $g$) of the pullback metric $u^*h$.
The associated Euler--Lagrange equations in a local coordinate system $(x^\mu)$ on $M$ read 
\begin{equation}
\label{eq:wm}
\Box_g u^a(x)+g^{\mu \nu}(x)\Gamma^a_{bc}(u(x))\partial_\mu u^b(x) \partial_\nu u^c(x)=0
\end{equation}
where Einstein's summation convention is assumed and $\Gamma^a_{bc}$ are the Christoffel symbols on $N$.
The system \eqref{eq:wm} is called the wave maps equation (in the intrinsic form).
In what follows we choose the base manifold $M$ to be Minkowski space.
Due to their geometric origin, wave maps are appealing models in nonlinear field theory which naturally generalize the linear wave equation.
They are relevant in various areas of physics but still simple enough to be accessible for rigorous mathematical analysis. 

Basic questions for a nonlinear time evolution equation are: Do there exist global (that is, for all times) solutions for all data or is it possible to specify initial data that lead to a breakdown of the solution in finite time?
If the latter is true, then how does the breakdown (blow up) occur?
There exists a heuristic principle which gives a hint for scaling invariant equations that possess a positive energy (such as the wave maps equation, see \cite{DSA}).
If the scaling behavior is such that shrinking of the solution is energetically favorable, then one expects finite time blow up.
Conversely, if shrinking to ever smaller scales is energetically forbidden, then global existence is anticipated.
These two cases are called energy supercritical and energy subcritical, respectively.
There is also a borderline situation where the energy itself is scaling invariant which is called energy critical.
For wave maps, the criticality class is linked to the spatial dimension of the base manifold $M$.
The equation is energy subcritical, critical or supercritical if $\dim M=1+1$, $\dim M=2+1$ or $\dim M \geq 3+1$, respectively.
It is fair to say that our current understanding of large data problems is confined to energy subcritical and critical equations.
However, this is a very unsatisfactory situation since many problems in physics turn out to be energy supercritical.

The initial value problem for the wave maps equation has attracted a lot of interest by the mathematical community in the past two decades which led to the development of new sophisticated tools that tremendously improved our understanding of nonlinear wave equations.
In particular the case where the base manifold $M$ is assumed to be Minkowski space (which we shall do 
throughout) 
has been studied thoroughly.
Since it is impossible for us to do justice to the huge amount of publications on the subject, we have to restrict ourselves to a certain selection of some important contributions. 
A considerable amount of the literature is devoted to problems with certain symmetry properties, such as radial or equivariant maps. We mention for instance \cite{Chr1}, \cite{Chr2}, \cite{shatah88}, \cite{STZ94},
\cite{struwe04a}, \cite{struwe04b} where various fundamental aspects like local and global well--posedness, asymptotic behavior or blow up are studied.
At the end of the 1990s new techniques were developed, capable of treating the full system without symmetry assumptions.
A main objective in this respect was to prove small data global existence for various target manifolds and space
dimensions, see e.g.,
\cite{struwe01}, \cite{keel-tao}, \cite{struwe99}, \cite{tataru99}, \cite{tataru01}, \cite{tao01a}, \cite{tao01b}, \cite{klainerman-rodnianski02}, 
\cite{shatah-struwe02}, \cite{tataru05}, \cite{krieger03}, \cite{krieger04}, \cite{nahmod02}, \cite{nahmod03}. 
On the other hand, for large data and in the energy critical case, there are newer results on blow up, e.g., 
\cite{struwe03}, \cite{KST08}, 
\cite{rodnianski-sterbenz06}, \cite{rodnianski-raphael09}, \cite{carstea}
and, very recently, also on global existence
\cite{krieger-schlag09}, \cite{tataru-sterbenz09a},
\cite{tataru-sterbenz09b}, \cite{tao09}. 
We will comment below in more detail on some of those works which are most relevant for us. 
Despite the already vast literature on the subject we are still only at the beginning and many questions remain open, even for highly symmetric problems.
We also refer the reader to the monograph \cite{struwe98} for a general introduction on the subject as well as the survey article \cite{kriegersurv} for an in--depth review of recent results and applications of wave maps.

In this paper we study the simplest energy supercritical case: co--rotational wave maps from $(3+1)$ Minkowski space to the three--sphere.
This model can be described by the single semilinear wave equation
\begin{equation}
\label{eq:main}
\psi_{tt}-\psi_{rr}-\frac{2}{r}\psi_r+\frac{\sin(2\psi)}{r^2}=0
\end{equation}
where $r \geq 0$ is the standard radial coordinate on Minkowski space. We refer to \cite{DSA} and references therein for more details on the underlying symmetry reduction which is a special case of equivariance.
Global well--posedness of the Cauchy problem for Eq.~\eqref{eq:main} for data which are small in a sufficiently high Sobolev space follows from \cite{sideris}.
Furthermore, a number of results concerning the Cauchy problem for equivariant wave maps are obtained in \cite{STZ94}, in particular, local well--posedness with minimal regularity requirements is studied.
Global well--posedness for data with small energy can be concluded from the much more general result in \cite{tao01b}.
On the other hand, it has been known for a long time that Eq.~\eqref{eq:main} exhibits finite time blow up in the form of self--similar solutions \cite{shatah88}, see also \cite{cazenave} for generalizations.
By definition, a self--similar solution is of the form $\psi(t,r)=f(\frac{r}{T-t})$ for a constant $T>0$ and obviously, depending on the concrete form of $f$, ``something singular'' happens as $t \to T-$.
As usual, by exploiting finite speed of propagation, a self--similar solution can be used to construct a solution with compactly supported initial data that breaks down in finite time.
In fact, Eq.~\eqref{eq:main} admits many self--similar solutions \cite{bizon00} and a particular one, henceforth denoted by $\psi^T$, is even known in closed form \cite{TS90} and given by
$$ \psi^T(t,r)=2 \arctan \left (\tfrac{r}{T-t} \right )=:f_0 \left (\tfrac{r}{T-t} \right ). $$
We call $\psi^T$ the \emph{ground state} or \emph{fundamental self--similar solution}.

We remark in passing that the self--similar blow up here is very different from singularity formation in the 
analogous model of equivariant wave maps on ($2+1$) Minkowski space.
For this energy critical problem it has been shown by Struwe \cite{struwe03} that the blow up (if it exists) takes place via shrinking of a harmonic map at a rate which is strictly faster than self--similar.
This beautiful result holds for a large class of targets and, by ruling out the existence of finite energy harmonic maps, it can be used to show global existence.
In the case of the two--sphere as a target, there do exist finite energy harmonic maps and indeed, blow up solutions for this model have been constructed in \cite{KST08}, \cite{rodnianski-sterbenz06}, \cite{rodnianski-raphael09}.

\subsection{The main result}
\label{sec:intromainresult}
Based on numerical investigations \cite{bizon99}, the solution $\psi^T$ is expected to be fundamental for the understanding of the dynamics of the system. 
This is due to the fact that it acts as an attractor in the sense that generic large data evolutions approach $\psi^T$ locally near the center $r=0$ as $t \to T-$.
Consequently, the blow up described by $\psi^T$ is expected to be stable.
In this paper we give a rigorous proof for the stability of $\psi^T$.
Our result is conditional in the sense that it depends on a certain spectral property for a linear ordinary differential equation which cannot be verified rigorously so far.
However, very reliable numerics and partial theoretical results leave no doubt that it is satisfied, see
\cite{DSA} for a thorough discussion of this issue.
In other words, our result reduces the question of (nonlinear) stability of $\psi^T$ to a linear ODE spectral problem.
At this point we should remark that some aspects of the (linear) stability problem for self--similar wave maps are studied
in \cite{ichpca2} by using a hyperbolic coordinate system adapted to self--similarity.
In particular stability properties of excited self--similar solutions are derived in \cite{ichpca2}.
However, as discussed in \cite{ichpca2}, the hyperbolic coordinate system is not suitable for proving 
stability of $\psi^T$. 

In order to formulate our main theorem, we need a few preparations.
Based on the known numerical results, one expects convergence to the self--similar attractor $\psi^T$ only in
 $$ \mc{C}_T:=\{(t,r): t\in (0,T), r \in [0,T-t]\}, $$
the backward lightcone of the blow up point.
Consequently, we study the Cauchy problem
\begin{equation}
 \label{eq:maincauchy}
\left \{ \begin{array}{l}
\psi_{tt}(t,r)-\psi_{rr}(t,r)-\frac{2}{r}\psi_r(t,r)+\frac{\sin(2\psi(t,r))}{r^2}=0 \mbox{ for } (t,r) \in \mc{C}_T \\
\psi(0,r)=f(r), \psi_t(0,r)=g(r) \mbox{ for }r \in [0,T]
         \end{array} \right .
\end{equation}
with given initial data $(f,g)$.
Our result applies to small perturbations $(f,g)$ of $(\psi^T(0,\cdot), \psi^T_t(0, \cdot))$, i.e., we view the nonlinear problem as a perturbation of the linearization of \eqref{eq:maincauchy} around $\psi^T$.
Writing 
\begin{equation} 
\label{eq:nonlinearity}
\sin(2(\psi^T+\varphi))=\sin(2\psi^T)+2\cos(2\psi^T)\varphi+N_T(\varphi) 
\end{equation}
with $N_T(\varphi)=O(\varphi^2)$ (if $\varphi$ is small), we obtain the equation
\begin{equation} 
\label{eq:mainlin}
\varphi_{tt}-\varphi_{rr}-\frac{2}{r}\varphi_r+\frac{2}{r^2}\varphi+\frac{2\cos(2\psi^T)-2}{r^2}\varphi+\frac{N_T(\varphi)}{r^2}=0 
\end{equation}
for perturbations $\varphi$ of $\psi^T$.
Note that the ``potential term'' in Eq.~\eqref{eq:mainlin} is time--dependent since $\psi^T$ is self--similar.
Consequently, it is convenient to remove this time dependence by switching to similarity coordinates $\tau:=-\log(T-t)$ and $\rho:=\frac{r}{T-t}$.
This transforms the lightcone $\mc{C}_T$ to an infinite cylinder and the blow up point is shifted towards $\infty$.
Thus, we obtain an asymptotic stability problem which is explicitly given by
\begin{equation}
\label{eq:cssscalar}
\phi_{\tau \tau}+\phi_\tau+2\rho \phi_{\tau \rho}-(1-\rho^2)\phi_{\rho \rho}-2\frac{1-\rho^2}{\rho}\phi_\rho+\frac{V(\rho)}{\rho^2}\phi+\frac{N_T(\phi)}{\rho^2}=0 
\end{equation}
where $\phi(\tau,\rho)=\varphi(T-e^{-\tau}, e^{-\tau}\rho)$
and the ``potential'' $V$ reads
\begin{equation} 
\label{eq:V}
V(\rho)=2 \cos(4 \arctan(\rho))=\frac{2(1-6 \rho^2+\rho^4)}{(1+\rho^2)^2}. 
\end{equation}
The relevant coordinate domain is $\tau \geq -\log T$ and $\rho \in [0,1]$.
The corresponding \emph{linearized problem} is simply obtained by ignoring the nonlinear term, i.e., 
\begin{equation}
\label{eq:linearcssscalar}
\phi_{\tau \tau}+\phi_\tau+2\rho \phi_{\tau \rho}-(1-\rho^2)\phi_{\rho \rho}-2\frac{1-\rho^2}{\rho}\phi_\rho+\frac{V(\rho)}{\rho^2}\phi=0. 
\end{equation}
Inserting the mode ansatz $\phi(\tau,\rho)=e^{\lambda \tau}u_\lambda(\rho)$ into Eq.~\eqref{eq:linearcssscalar}, we arrive at the aforementioned linear ODE spectral problem
\begin{equation}
\label{eq:evodeintro}
-(1-\rho^2)u_\lambda''(\rho)-2\frac{1-\rho^2}{\rho}u_\lambda'(\rho)+2\lambda \rho u_\lambda'(\rho)+\lambda (\lambda+1)u_\lambda(\rho)+\frac{V(\rho)}{\rho^2}u_\lambda(\rho)=0
\end{equation}
for the function $u_\lambda$.
We say that $\lambda$ is an eigenvalue if Eq.~\eqref{eq:evodeintro} has a solution $u_\lambda \in C^\infty[0,1]$. It is a consequence of \cite{DSA} that only smooth solutions are relevant here.
Furthermore, the eigenvalue is said to be unstable if $\mathrm{Re}\lambda \geq 0$ and stable if $\mathrm{Re}\lambda<0$.
It can be immediately checked that $\lambda=1$ is an (unstable) eigenvalue with 
$u_1(\rho)=\frac{\rho}{1+\rho^2}$.
However, it turns out that this instability is an artefact of the similarity coordinates and it does not correspond to a ``real'' instability of the solution $\psi^T$.
In fact, it is a manifestation of the time translation invariance of the wave maps equation.
Consequently, we say that $\psi^T$ is \emph{mode stable} if $\lambda=1$ is the only unstable eigenvalue.
The mode stability of $\psi^T$, which we shall assume here, has been verified numerically using various independent techniques, see \cite{bizon99}, \cite{ichdipl}, \cite{ich0} and \cite{bizon05}.
Furthermore, in \cite{ichpca2} it is rigorously proved that $\lambda=1$ is the only eigenvalue with $\mathrm{Re}\lambda
\geq 1$ and in \cite{ichpca1} we show that there do not exist \emph{real} unstable eigenvalues (except $\lambda=1$).
Finally, in \cite{DSA} it is shown that $\lambda=1$ is the only eigenvalue with real part greater than $\frac12$.
All these results leave no doubt that $\psi^T$ is indeed mode stable although a completely 
rigorous proof of this property is still not available.
At this point we also mention that according to numerics \cite{bizon99}, \cite{ichdipl}, \cite{ich0} and \cite{bizon05}, the first stable eigenvalue is $\approx -0.54$. This is important since it \emph{dictates
the rate of convergence} to the self--similar solution, see Theorem \ref{thm:main} below.

For pairs of functions $(f,g) \in C^2[0,R] \times C^1[0,R]$, $R>0$, that satisfy the boundary condition $f(0)=g(0)=0$ we introduce a norm $\|\cdot\|_{\mc{E}(R)}$ by setting
$$ \|(f,g)\|_{\mc{E}(R)}^2:=
\int_0^R |r f''(r)+3 f'(r)|^2 dr+\int_0^R |r g'(r)+2 g(r)|^2 dr. $$
Observe that $\|\cdot\|_{\mc{E}(R)}$ is indeed a norm since $r f''(r)+3f'(r)=0$ implies 
$f(r)=c_1+\frac{c_2}{r^2}$ for constants $c_1, c_2$ but this function does not belong to $C^2[0,R]$ unless $c_2=0$ and it does not satisfy the boundary condition $f(0)=0$ unless $c_1=0$.
A similar statement is true for $g$.
The motivation for the choice of the norm $\|\cdot\|_{\mc{E}(R)}$ is two--fold.
First, it is derived from a conserved quantity of the \emph{free equation}
\begin{equation} 
\label{eq:free}
\varphi_{tt}-\varphi_{rr}-\frac{2}{r}\varphi_r+\frac{2}{r^2}\varphi=0 
\end{equation}
which is obtained from Eq.~\eqref{eq:mainlin} by dropping the regularized ``potential'' and the nonlinearity.
More precisely, for any sufficiently regular solution $\varphi$ of Eq.~\eqref{eq:free}, the function $t \mapsto \|(\varphi(t,\cdot),\varphi_t(t,\cdot))\|_{\mc{E}(\infty)}$ is a constant.
This can be seen as follows.
Suppose $\varphi$ is a (sufficiently smooth) solution of Eq.~\eqref{eq:free} and set
\begin{equation}
\label{eq:phihat} 
\hat{\varphi}(t,r):=r\varphi_r(t,r)+2\varphi(t,r). 
\end{equation}
It is important to note here that this transformation is invertible.
Indeed, we have
$$ r\hat{\varphi}(t,r)=\partial_r (r^2 \varphi(t,r)) $$
and this necessarily implies
$$ \varphi(t,r)=\frac{1}{r^2}\int_0^r r' \hat{\varphi}(t,r')dr' $$
since we assume regularity of $\varphi$ at the origin.
Now note that 
\begin{align*}
\hat{\varphi}_{tt}-\hat{\varphi}_{rr}&=\tfrac{1}{r}\partial_r(r^2 \varphi_{tt})-r\varphi_{rrr}
-4\varphi_{rr} \\
&=\tfrac{1}{r}\partial_r r^2 \left [\varphi_{tt}-\varphi_{rr}-\tfrac{2}{r}\varphi_r
+\tfrac{2}{r^2}\varphi \right ]=0
\end{align*}
and thus, $\hat{\varphi}$ satisfies the one--dimensional wave equation on the half--line $r \geq 0$.
Furthermore, since $\varphi(t,0)=0$ by regularity, we have $\hat{\varphi}(t,0)=0$ for all $t$ and thus, 
$$ \int_0^\infty \left [\hat{\varphi}_t^2(t,r)+\hat{\varphi}_r^2(t,r) \right ]dr
=\|(\varphi(t,\cdot), \varphi_t(t,\cdot))\|_{\mc{E}(\infty)}^2 $$
is independent of $t$.
Consequently, $\|\cdot\|_{\mc{E}(R)}$ is a local ``higher energy norm'' for the free equation \eqref{eq:free} since it requires one more derivative than the energy.
The point of requiring more derivatives is that one can ``see'' self--similar blow up in this norm, which is the second important feature of $\|\cdot\|_{\mc{E}(R)}$. 
Explicitly, we have
\begin{align*} 
\|(\psi^T(t,\cdot),\psi_t^T(t,\cdot))\|_{\mc{E}(T-t)}^2&=\int_0^{T-t}|r\psi_{rr}^T(t,r)+3\psi_r^T(t,r)|^2 dr
\\
&\quad +\int_0^{T-t}|r\psi_{tr}^T(t,r)+2\psi_t^T(t,r)|^2 dr \\
&=\int_0^{T-t} \left |\tfrac{r}{(T-t)^2}f_0''\left (\tfrac{r}{T-t} \right )+\tfrac{3}{T-t}f_0'\left (\tfrac{r}{T-t} \right ) \right |^2 dr \\
&\quad +\int_0^{T-t}\left |\tfrac{r^2}{(T-t)^3}f_0''\left (\tfrac{r}{T-t} \right )
+\tfrac{3r}{(T-t)^2}f_0'\left (\tfrac{r}{T-t} \right ) \right |^2 dr \\
&=\frac{1}{T-t}\left [\int_0^1 |\rho f_0''(\rho)+3f_0'(\rho)|^2 d\rho+\int_0^1 |\rho^2 f_0''(\rho)+3 \rho f_0'(\rho)|^2 d\rho \right ],
\end{align*}
and thus,
\begin{equation}
\label{eq:blowup} 
\|(\psi^T(t,\cdot),\psi_t^T(t,\cdot))\|_{\mc{E}(T-t)}=C(T-t)^{-\frac{1}{2}} 
\end{equation}
for a constant $C>0$ which shows that the norm of $\psi^T$ blows up.
Note carefully that this is in stark contrast to the behavior of the energy norm. 
The local energy of $\psi^T$ in the cone $\mc{C}_T$ \emph{decays} like $(T-t)$ as $t \to T-$ and thus, the blow up is ``invisible'' in the energy norm. 
This is a manifestation of the energy supercritical character of the equation.
As a consequence, we have to work in a stronger topology, see also \cite{DSA} for a discussion on this issue.

Finally, for initial data $(f,g) \in C^3[0,\frac{3}{2}] \times C^2[0,\frac{3}{2}]$ with $f(0)=g(0)=0$, we define another norm $\|\cdot\|_{\mc{E}'}$ by
\begin{align*}
 \|(f,g)\|_{\mc{E}'}^2:=&\int_0^{3/2} |rf'''(r)+4f''(r)|^2 r^2 dr+\int_0^{3/2} |rf''(r)+3f'(r)|^2 dr \\
 &+\int_0^{3/2} |r^2 g''(r)+4r g'(r)+2 g(r)|^2 dr.
\end{align*}
The norm $\|\cdot\|_{\mc{E}'}$ is stronger than $\|\cdot\|_{\mc{E}(\frac{3}{2})}$ in the sense that it requires one additional derivative. 
This stronger norm is needed for certain technical reasons which will become clear below. 
Now we are ready to formulate our main result.

\begin{theorem}
 \label{thm:main}
 Assume that $\psi^T$ is mode stable and denote by $s_0$ the real part of the first 
\footnote{That is, the stable eigenvalue with the largest real part.}
stable eigenvalue.
 Let $\varepsilon>0$ be arbitrary but so small that $\omega:=\max\{-\frac{1}{2},s_0\}+\varepsilon<0$.
 Furthermore, let $(f,g) \in C^3[0,\frac{3}{2}] \times C^2[0,\frac{3}{2}]$ be initial data with $f(0)=g(0)=0$ and
 $$ \|(f,g)-(\psi^1(0,\cdot),\psi_t^1(0,\cdot))\|_{\mc{E}'}<\delta $$
 for a sufficiently small $\delta>0$.
 Then there exists a $T>0$ close to $1$ such that the Cauchy problem
 \begin{equation*}
\left \{ \begin{array}{l}
\psi_{tt}(t,r)-\psi_{rr}(t,r)-\frac{2}{r}\psi_r(t,r)+\frac{\sin(2\psi(t,r))}{r^2}=0, \quad t \in (0,T), 
r \in [0,T-t] \\
\psi(0,r)=f(r), \psi_t(0,r)=g(r), \quad r \in [0,T]
         \end{array} \right .
\end{equation*}
has a unique solution $\psi$ that satisfies
$$ \|(\psi(t,\cdot),\psi_t(t,\cdot))-(\psi^T(t,\cdot),\psi_t^T(t,\cdot))\|_{\mc{E}(T-t)} \leq C_\varepsilon
|T-t|^{-\frac{1}{2}+|\omega|} $$
for all $t \in [0,T)$ where $C_\varepsilon>0$ is a constant that depends on $\varepsilon$. 
\end{theorem}

Several remarks are in order.
 \begin{itemize}
 \item As usual, by a ``solution'' $\psi$ we mean a function that solves the equation in an appropriate weak sense and not necessarily in the sense of classical derivatives. 
  \item According to Eq.~\eqref{eq:blowup}, one should actually normalize the estimate from Theorem \ref{thm:main} to
  $$ |T-t|^{\frac{1}{2}}\|(\psi(t,\cdot),\psi_t(t,\cdot))-(\psi^T(t,\cdot),\psi_t^T(t,\cdot))\|_{\mc{E}(T-t)} \leq C_\varepsilon
|T-t|^{|\omega|} $$
for $t \in [0,T)$ which shows that the solution $\psi$ converges to $\psi^T$ in the backward lightcone of the blow up point $(T,0)$.
Consequently, Theorem \ref{thm:main} tells us that, if we start with initial data that are sufficiently close to $(\psi^1(0,\cdot),\psi_t^1(0,\cdot))$, then the solution $\psi$ blows up in a self--similar manner via $\psi^T$.
In this sense, the blow up described by $\psi^T$ is stable.
It is clear that even very small (generic) perturbations of the initial data $(\psi^1(0, \cdot),\psi_t^1(0,\cdot))$ will change the blow up time of the solution. 
That is why one has to adjust $T$.
\item It should be emphasized that the rate of convergence to the attractor is dictated by the 
first stable eigenvalue. This complies with naive expectations and previous heuristics and numerics 
in the physics literature, e.g., \cite{bizon99}.
\item The radius $\frac{3}{2}$ in the smallness condition for the initial data
is more or less an arbitrary choice. The problem is that one actually needs 
to prescribe the data on the interval $[0,T]$
but $T$ is not known in advance --- a tautology.
However, since our argument yields a $T$ close to $1$, one may assume that $T$ is always smaller than $\frac{3}{2}$.

\item We have to require one more derivative of the initial data than we actually control in the time evolution. This is for technical reasons.

\item As usual, one does not really need classical derivatives of the data $(f,g)$ --- weak derivatives are fine too, as long as the norm $\|(f,g)\|_{\mc{E}'}$ is well--defined.
In other words, the initial data may be taken from the completion of the space 
$$ \left \{(f,g)\in C^3[0,\tfrac{3}{2}] \times C^2[0,\tfrac{3}{2}]: f(0)=g(0)=0 \right \} $$ 
with respect to the norm $\|\cdot\|_{\mc{E}'}$.
It is worth noting here that the norm $\|\cdot\|_{\mc{E}'}$ is strong enough for the boundary conditions to ``survive'', cf.~Lemma \ref{lem:Hcont} below.

\item It seems to be difficult to rigorously \emph{prove} mode stability of the ground state self--similar solution. 
However, this is a linear ODE problem and there are very reliable numerics as well as partial rigorous results that leave no doubt that $\psi^T$ is mode stable (cf.~\cite{DSA}, Sec.~3). 
Nevertheless, it would be desirable to have a proof for the mode stability and we plan to revisit this problem elsewhere.
\end{itemize}

\subsection{Outline of the proof}
The proof of Theorem \ref{thm:main} is functional analytic and the main tools we use are the linear perturbation theory developed in the companion paper \cite{DSA} as well as the implicit function theorem on Banach spaces.
The implicit function theorem is used to ``push'' the linear results from \cite{DSA} to the nonlinear level. Consequently, our approach is perturbative in the sense that the nonlinearity is viewed as a perturbation of the linear problem.

In a first step, we write the Cauchy problem in Theorem \ref{thm:main} as an ODE on a suitable Hilbert space of the form
\begin{equation}
\label{eq:maincssopintro}
 \left \{ \begin{array}{l}
\frac{d}{d \tau}\Phi(\tau)=L\Phi(\tau)+\mb{N}(\Phi(\tau)) \mbox{ for }\tau>-\log T \\
\Phi(-\log T)=\mb{U}(\mb{v},T)
          \end{array} \right .
\end{equation}
which is, in fact, an operator formulation of Eq.~\eqref{eq:cssscalar} since we are working in similarity coordinates $(\tau,\rho)$.
The field $\Phi$ describes a nonlinear perturbation of $\psi^T$.
Here, the linear operator $L$ emerges from the linearization of the equation around the fundamental self--similar solution $\psi^T$.
The properties of $L$ have been studied thoroughly in \cite{DSA}.
Furthermore, the nonlinear operator $\mb{N}$ results from the nonlinear remainder of the equation and the free data $(f,g)$ from Theorem \ref{thm:main} define the vector $\mb{v}$.
In fact, $\mb{v}$ are the data relative to $\psi^1$.
Up to some variable transformations we essentially have $\mb{v} \approx (f,g)-\psi^1[0]$ and
$$\mb{U}(\mb{v},T) \approx \mb{v}+\psi^1[0]-\psi^T[0]=(f,g)-\psi^T[0]$$ where
$\psi[0]$ is shorthand for $(\psi(0,\cdot),\psi_t(0,\cdot))$.
Consequently, the smallness condition in Theorem \ref{thm:main} ensures that $\mb{v}$ is small.
We apply Duhamel's formula and time translation in order to derive an integral equation for solutions of Eq.~\eqref{eq:maincssopintro} of the form
\begin{equation}
\label{eq:maincssopmildintro}
 \Psi(\tau)=S(\tau)\mb{U}(\mb{v},T)+\int_0^\tau S(\tau-\tau')\mb{N}(\Psi(\tau'))d\tau', \quad \tau \geq 0
\end{equation}  
where $\Psi(\tau)=\Phi(\tau-\log T)$.
Here, $S$ is the semigroup generated by $L$, i.e., the solution operator to the linearized problem.
The existence of $S$ has been proved in \cite{DSA} and, moreover, it has been shown that there exists a projection $P$ with one--dimensional range that commutes with $S(\tau)$ such that $S$ satisfies the estimates
$$ \|S(\tau)P\|\lesssim e^{\tau}, \quad \|S(\tau)(1-P)\|\lesssim e^{-|\omega|\tau} $$
for all $\tau \geq 0$ with $\omega$ from Theorem \ref{thm:main}, see Theorem \ref{thm:linear} below.
In other words, the subspace $\mathrm{rg}P$ is unstable for the linear time evolution.
In order to compensate for this instability, we first construct a solution 
to a modified equation
\begin{align*} \Psi(\tau)=&S(\tau)\mb{U}(\mb{v},T)-e^\tau P \left [\mb{U}(\mb{v},T)+\int_0^\infty e^{-\tau'}\mb{N}(\Psi(\tau'))d\tau' \right ] \\
&+\int_0^\tau S(\tau-\tau')\mb{N}(\Psi(\tau'))d\tau', \quad \tau \geq 0.
\end{align*}
The ``correction''
$$ \mb{F}(\mb{v},T):=P \left [\mb{U}(\mb{v},T)+\int_0^\infty e^{-\tau'}\mb{N}(\Psi(\tau'))d\tau' \right ] $$
suppresses the instability of the solution.
We remark that some aspects of this construction are inspired by the work of Krieger and Schlag on the critical wave equation \cite{schlag}.
An application of the implicit function theorem yields the existence of a solution $\Psi$ to the modified equation that decays like $e^{-|\omega|\tau}$ provided the data $\mb{v}$ are small and $T$ is sufficiently close to $1$.
Thus, we retain the linear decay on the nonlinear level.
In a second step, we show that, for given small $\mb{v}$, there exists a $T$ close to $1$ such that $\mb{F}(\mb{v},T)$ is in fact zero and thereby, we obtain a decaying solution to Eq.~\eqref{eq:maincssopmildintro}.
This is again accomplished by an application of the implicit function theorem.
Here, it is crucial that the Fr\'echet derivative $\partial_T \mb{U}(\mb{0},T)|_{T=1}$ is an element of the unstable subspace $\mathrm{rg}P$.
In other words, the unstable subspace is spanned by the tangent vector at $T=1$ to the curve $T \mapsto \mb{U}(\mb{0},T) \approx \psi^T[0]$ in the initial data space. 
This curve is generated by varying the blow up time $T$ and
therefore, the instability of the linear evolution is caused by the time translation symmetry of the wave maps equation.

\subsection{Notations and conventions}
For Banach spaces $X,Y$ we denote by $\mc{B}(X,Y)$ the Banach space of bounded linear operators from $X$ to $Y$. As usual, we write $\mc{B}(X)$ if $X=Y$.
For a Fr\'echet differentiable map $F: U \times V \subset X \times Y \to Z$ where $X,Y,Z$ are Banach spaces and $U \subset X$, $V \subset Y$ are open subsets, we denote by $D_1 F: U \times V \to \mc{B}(X,Z)$, $D_2 F: U\times V \to \mc{B}(Y,Z)$ the respective partial Fr\'echet derivatives and by $DF$ the (full) Fr\'echet derivative of $F$.
Vectors are denoted by bold letters and the individual components are numbered by lower indices, e.g., $\mb{u}=(u_1,u_2)$. We do not distinguish between row and column vectors.
Furthermore, for $a,b \in \mathbb{R}$ we use the notation $a \lesssim b$ if there exists a constant $c>0$ such that $a \leq cb$ and we write $a \simeq b$ if $a \lesssim b$ and $b \lesssim a$. 
Unless otherwise stated, it is implicitly assumed that the constant $c$ is absolute, i.e., it does not depend on any of the involved quantities in the inequality.
The symbol $\sim$ is reserved for asymptotic equality.
Finally, the letter $C$ (possibly with indices) denotes a generic nonnegative constant which is not supposed to have the same value at each occurrence. 

\section{Review of the linear perturbation theory}

For the convenience of the reader we review the results recently obtained in \cite{DSA} which prepare the ground for the nonlinear stability theory of the self--similar wave map $\psi^T$ elaborated in the present paper.

\subsection{First order formulation}
Roughly speaking, the main outcome of our paper \cite{DSA} is a well--posed initial value formulation of the linear evolution problem \eqref{eq:linearcssscalar}.
In order to accomplish this, the equation \eqref{eq:mainlin} is first transformed to a first--order system in time by defining the new variables
\begin{align*} 
\varphi_1(t,r)&:=\frac{r^2 \varphi_t(t,r)}{T-t} \\
\varphi_2(t,r)&:=r\varphi_r(t,r)+2 \varphi(t,r).
\end{align*}
This is motivated by the discussion in the introduction, observe in particular that $\varphi_2$ is nothing but $\hat{\varphi}$ in Eq.~\eqref{eq:phihat}.
The variable $\varphi_1$ is an appropriately scaled time derivative.
As before, the inverse transformation is 
$$ \varphi(t,r)=\frac{1}{r^2}\int_0^r r' \varphi_2(t,r')dr' $$
and by definition, the free operator is given by
$$ \varphi_{rr}+\tfrac{2}{r}\varphi_r-\tfrac{2}{r^2}\varphi=\tfrac{1}{r}\partial_r \varphi_2
-\tfrac{1}{r^2}\varphi_2. $$ 
Consequently, we obtain
\begin{align}
\label{eq:sys1}
\partial_t \varphi_1&=\frac{r^2 \varphi_{tt}}{T-t}+\frac{r^2 \varphi_t}{(T-t)^2} \\
&=\frac{r^2}{T-t}\left [\varphi_{rr}+\frac{2}{r}\varphi_r-\frac{2}{r^2}\varphi-\frac{2\cos(2\psi^T)-2}{r^2}\varphi-\frac{N_T(\varphi)}{r^2} \right ]+\frac{\varphi_1}{T-t} \nonumber \\
&=\frac{1}{T-t}\left [ \varphi_1+r\partial_r \varphi_2-\varphi_2-\frac{2\cos(2\psi^T)-2}
{r^2}\int r \varphi_2-N_T \left (\tfrac{1}{r^2}\int r \varphi_2 \right ) \right ] \nonumber
\end{align}
where $\int r \varphi_2$ is shorthand for $\int_0^r r' \varphi_2(t,r')dr'$ and, as the reader has already noticed, we are a bit sloppy with our notation and omit the arguments of the functions occasionally.
The second equation in the first--order system is simply given by the identity
\begin{equation}
\label{eq:sys2}
\partial_t \varphi_2=\tfrac{1}{r}\partial_r (r^2 \varphi_t)=\frac{T-t}{r}\partial_r \varphi_1. 
\end{equation}
Furthermore, the initial data $(\psi(0,\cdot),\psi_t(0,\cdot))=(f,g)$ from Eq.~\eqref{eq:maincauchy} 
transform into
\begin{align}
\label{eq:sys3}
\varphi_1(0,r)&=\tfrac{r^2}{T}\left [g(r)-\psi_t^T(0,r) \right ] \\
\varphi_2(0,r)&=r \left [f'(r)-\psi_r^T(0,r) \right ]+2 \left [f(r)-\psi^T(0,r) \right ]. \nonumber
\end{align}
Now we rewrite Eqs.~\eqref{eq:sys1}, \eqref{eq:sys2} and \eqref{eq:sys3} in similarity coordinates 
$$\tau=-\log(T-t),\quad \rho=\frac{r}{T-t}$$ by setting 
$$ \phi_j(\tau,\rho):=\varphi_j(T-e^{-\tau}, e^{-\tau}\rho), \quad j=1,2 $$
and recalling that $\partial_t=e^\tau (\partial_\tau+\rho \partial_\rho)$, $\partial_r=e^\tau \partial_\rho$.
Furthermore, note that
$$ \frac{1}{r^2}\int_0^r r'\varphi_2(t,r')dr'=\frac{1}{\rho^2}\int_0^\rho \rho \phi_2(t,\rho')d\rho' $$
and thus, we arrive at the system
\begin{equation}
\label{eq:main1stcss}
\left \{ \begin{array}{l}
\left. \begin{array}{l}
\partial_\tau \phi_1=-\rho \partial_\rho \phi_1+\phi_1+\rho \partial_\rho \phi_2
-\phi_2-\frac{V(\rho)-2}{\rho^2}\int \rho \phi_2-N_T \left ( \rho^{-2}\int \rho \phi_2 \right ) 
 \\
\partial_\tau \phi_2=\frac{1}{\rho}\partial_\rho \phi_1-\rho \partial_\rho \phi_2 \end{array} \right \}
\mbox{ in } \mc{Z}_T \\
\left. \begin{array}{l}
\phi_1(-\log T,\rho)=T\rho^2\left [g(T\rho)-\psi^T_t(0,T\rho) \right ] \\
\phi_2(-\log T,\rho)=T\rho \left [f'(T\rho)-\psi^T_r(0,T\rho) \right ] +2 \left [f(T\rho)-\psi^T(0,T\rho)
\right ] 
\end{array} \right \}
\mbox{ for } \rho \in [0,1]
\end{array} \right .
\end{equation} 
where $\mc{Z}_T:=\{(\tau,\rho): \tau > -\log T, \rho \in [0,1]\}$ and $V$ is given by Eq.~\eqref{eq:V}.
The system \eqref{eq:main1stcss} is equivalent to the original Cauchy problem \eqref{eq:maincauchy}.
We again recall that the orginal wave map $\psi$ can be reconstructed from $\phi_2$ by the formula
\begin{equation} 
\label{eq:origpsi}
\psi(t,r)=\psi^T(t,r)+\tfrac{1}{r^2}\int_0^r r' \phi_2\left (-\log(T-t), \tfrac{r'}{T-t} \right )dr'. 
\end{equation}
Furthermore, for the time derivative of the original field we have
\begin{equation}
\label{eq:origpsit}
\psi_t(t,r)=\psi_t^T(t,r)+\tfrac{T-t}{r^2}\phi_1\left (-\log(T-t),\tfrac{r}{T-t} \right ).
\end{equation}

In \cite{DSA} we have studied the corresponding \emph{linearized problem}, that is,
\begin{equation}
\label{eq:main1stcsslin}
\left \{ \begin{array}{l}
\left. \begin{array}{l}
\partial_\tau \phi_1=-\rho \partial_\rho \phi_1+\phi_1+\rho \partial_\rho \phi_2
-\phi_2-\frac{V(\rho)-2}{\rho^2}\int \rho \phi_2 
 \\
\partial_\tau \phi_2=\frac{1}{\rho}\partial_\rho \phi_1-\rho \partial_\rho \phi_2 \end{array} \right \}
\mbox{ in } \mc{Z}_T \\
\left. \begin{array}{l}
\phi_1(-\log T,\rho)=T\rho^2\left [g(T\rho)-\psi^T_t(0,T\rho) \right ] \\
\phi_2(-\log T,\rho)=T\rho \left [f'(T\rho)-\psi^T_r(0,T\rho) \right ] +2 \left [f(T\rho)-\psi^T(0,T\rho)
\right ] 
\end{array} \right \}
\mbox{ for } \rho \in [0,1]
\end{array} \right .
\end{equation} 
which is obtained from \eqref{eq:main1stcss} by dropping the nonlinearity.
The concrete form of the initial data is irrelevant on the linear level and therefore, we will omit it in this section.

\subsection{Operator formulation and well--posedness of the linearized problem}
\label{sec:oplin}
The basic philosophy of both the present paper and \cite{DSA}, is to formulate the evolution problems \eqref{eq:main1stcss} and \eqref{eq:main1stcsslin} as ordinary differential equations on suitable Hilbert spaces.
We would like to emphasize that \emph{the following choice of a Hilbert space is central and of crucial importance for the whole construction}.
\begin{definition}
We define the Hilbert space $\mc{H}$ as the completion of
$$ \{u \in C^2[0,1]: u(0)=u'(0)=0\} \times \{u \in C^1[0,1]: u(0)=0\} $$
with respect to the norm $\|\cdot\|:=\sqrt{(\cdot|\cdot)}$ induced by the inner product
$$ (\mb{u}|\mb{v}):=\int_0^1 u_1'(\rho)\overline{v_1'(\rho)}\frac{d\rho}{\rho^2}+\int_0^1 u_2'(\rho)\overline{v_2'(\rho)}d\rho. $$
\end{definition}
In order to motivate the boundary conditions, recall that 
$$ \varphi_1(t,r)=\frac{r^2 \varphi_t(t,r)}{T-t}, \quad \varphi_2(t,r)=r\varphi_r(t,r)+2\varphi(t,r) $$
and sufficiently regular solutions $\varphi$ of Eq.~\eqref{eq:mainlin} satisfy $\varphi_t(t,r)=O(r)$ as $r \to 0+$.
The Hilbert space $\mc{H}$ is supposed to hold the solution $(\phi_1, \phi_2)$ of Eq.~\eqref{eq:main1stcss} (or Eq.~\eqref{eq:main1stcsslin}) 
at a given instance of time and, since $\phi_j$ is nothing but $\varphi_j$ in similarity coordinates, the boundary conditions are natural requirements that follow from the behavior of regular solutions of Eq.~\eqref{eq:main} at the origin.
Note also that the norm $\|\cdot\|$ corresponds exactly to the local energy of $\hat{\varphi}$ as discussed in the introduction, cf.~Eq.~\eqref{eq:phihat}. 
We recall a convenient density result.
\begin{lemma}
\label{lem:Hdense}
The set $C^\infty_c(0,1] \times C^\infty_c(0,1]$ is dense in $\mc{H}$.
\end{lemma}

\begin{proof}
See \cite{DSA}.
\end{proof}

Furthermore, the space $\mc{H}$ is continuously embedded in $C[0,1] \times C[0,1]$.
\begin{lemma}
\label{lem:Hcont}
Let $\mb{u} \in \mc{H}$. Then $\mb{u} \in C[0,1] \times C[0,1]$ and we have the estimate
$$ \|\mb{u}\|_{L^\infty(0,1) \times L^\infty(0,1)} \lesssim \|\mb{u}\| $$
for all $\mb{u} \in \mc{H}$.
In particular, every $\mb{u} \in \mc{H}$ satisfies the boundary condition $\mb{u}(0)=\mb{0}$.
\end{lemma}

\begin{proof}
 See \cite{DSA}.
\end{proof}

Next we define a differential operator $\tilde{L}: \mc{D}(\tilde{L}) \subset \mc{H} \to \mc{H}$ on a 
suitable dense domain $\mc{D}(\tilde{L})$ by setting
$$ \tilde{L}\mb{u}(\rho):=\left (\begin{array}{c}
-\rho u_1'(\rho)+u_1(\rho)+\rho u_2'(\rho)-u_2(\rho)-V_1(\rho)\int_0^\rho \rho' u_2(\rho')d\rho' \\
\frac{1}{\rho}u_1'(\rho)-\rho u_2'(\rho) 
                       \end{array} \right ) $$
where the \emph{regularized potential} $V_1$ is given by
$$ V_1(\rho):=\frac{V(\rho)-2}{\rho^2}=-\frac{16}{(1+\rho^2)^2}. $$ 
We remark that in \cite{DSA} we actually start with a \emph{free operator} $\tilde{L}_0$.
In the notation of \cite{DSA} we have $\tilde{L}=\tilde{L}_0+L'$ where $L'$ contains the potential term and $L' \in \mc{B}(\mc{H})$.
Consequently, 
\begin{equation}
 \label{eq:linoppre}
 \left \{ \begin{array}{l}
\frac{d}{d\tau}\Phi(\tau)=\tilde{L}\Phi(\tau) \mbox{ for }\tau>-\log T \\
\Phi(-\log T)=\mb{u}
          \end{array} \right .
\end{equation}
is an operator formulation of \eqref{eq:main1stcsslin}
with $\Phi: [-\log T, \infty) \to \mc{H}$ and $\mb{u}$ are the initial data.
We have the following well--posedness result from \cite{DSA}.
\begin{proposition}
\label{prop:gen}
 The operator $\tilde{L}$ is closable and its closure $L$ generates a strongly continuous one--parameter semigroup $S: [0,\infty) \to \mc{B}(\mc{H})$ satisfying
 $$ \|S(\tau)\|\leq e^{(-\frac{1}{2}+\|L'\|)\tau} $$
 for all $\tau \geq 0$.
 In particular, the Cauchy problem
\begin{equation*}
 \left \{ \begin{array}{l}
\frac{d}{d\tau}\Phi(\tau)=L\Phi(\tau) \mbox{ for }\tau>-\log T \\
\Phi(-\log T)=\mb{u}
          \end{array} \right .
\end{equation*}
for $\mb{u} \in \mc{D}(L)$ has a unique solution which is given by
$$ \Phi(\tau)=S(\tau+\log T)\mb{u} $$
for all $\tau \geq -\log T$.
\end{proposition}
The proof of Proposition \ref{prop:gen} consists of an application of the Lumer--Phillips Theorem of semigroup theory, see e.g., \cite{engel}.
The rest of \cite{DSA} is concerned with a detailed spectral analysis of $L$ in order to improve the growth bound for $S(\tau)$.

As already mentioned in the introduction, the linearized problem Eq.~\eqref{eq:linearcssscalar} has an exponentially growing solution which emerges from the time translation symmetry of the original equation.
In the operator formulation, this is reflected by an unstable eigenvalue $\lambda=1$ in the spectrum of $L$.
The associated eigenspace is one--dimensional and spanned by the \emph{gauge mode} 
$$ \mb{g}(\rho):=\frac{1}{(1+\rho^2)^2}\left ( \begin{array}{c} 
2 \rho^3 \\ \rho(3+\rho^2) \end{array} \right ). $$
In order to formulate the main result of \cite{DSA}, we need to recall the definition of the spectral bound $s_0$.
\begin{definition}
\label{def:sb}
 The spectral bound $s_0$ is defined as 
 $$ s_0:=\sup\{\Re \lambda, \lambda \not=1: \mbox{ Eq.~\eqref{eq:evodeintro} has a nontrivial solution } 
u_\lambda \in C^\infty[0,1] \}. $$
\end{definition}
In other words, $s_0$ is the real part of the largest \footnote{The ``largest'' eigenvalue here means the eigenvalue with the largest real part.} eigenvalue (in the sense of Sec.~\ref{sec:intromainresult}) apart from the gauge eigenvalue $\lambda=1$.
We remark that in \cite{DSA} it is proved that $s_0<\frac{1}{2}$, however, numerical results (cf.~\cite{DSA}, Sec.~3) strongly suggest that in fact $s_0 \approx -0.54$.
The main result from \cite{DSA} gives a satisfactory description of the linearized time evolution and, via $s_0$, establishes the link between the linear Cauchy problem \eqref{eq:main1stcsslin} and the mode stability ODE \eqref{eq:evodeintro}.
\begin{theorem}[Mode stability implies linear stability, \cite{DSA}]
\label{thm:linear}
Let $\varepsilon>0$ be arbitrary. 
There exists a projection $P \in \mc{B}(\mc{H})$ onto $\langle \mb{g} \rangle$
which commutes with the semigroup $S(\tau)$ (and $PL \subset LP$) such that
$$ \|S(\tau)P\mb{f}\|=e^\tau \|P\mb{f}\| $$ as well as
$$ \|S(\tau)(1-P)\mb{f}\| \leq C_\varepsilon e^{(\max\{-
\frac{1}{2},s_0\}+\varepsilon)\tau}\|(1-P)\mb{f}\| $$
for all $\tau \geq 0$ and all $\mb{f} \in \mc{H}$ where $C_\varepsilon>0$ is a constant that depends on $\varepsilon$.
\end{theorem}
This important result tells us that the fundamental self--similar solution $\psi^T$ is linearly stable if one ``ignores'' the gauge instability and assumes that it is mode stable.
We remark that the proof of Theorem \ref{thm:linear} is highly nontrivial since the operator $L$ is not normal. 
This inconvenient fact introduces a number of technical difficulties one has to overcome.
In order to prove the existence of the spectral projection $P$ one first shows that $\lambda=1$ (the gauge eigenvalue) is isolated in the spectrum of $L$. This follows from a compactness argument.
Then the existence of $P$ is a consequence of the standard formula
$$ P=\frac{1}{2\pi i}\int_\Gamma (\lambda-L)^{-1}d\lambda $$
where $\Gamma$ is a suitable curve in the resolvent set of $L$ that encloses the point $1$.
However, since $L$ is not normal, one also has to deal with the possibility that the algebraic multiplicity of the gauge eigenvalue is larger than one.
This has to be excluded by ODE analysis of the eigenvalue equation.
Finally, the semigroup bounds on the stable subspace follow from a general abstract result.
We refer the reader to \cite{DSA} where the proof is carefully elaborated in full detail.

\section{The nonlinear perturbation theory}
In this section we study the full nonlinear problem \eqref{eq:main1stcss}.
As already remarked, \eqref{eq:main1stcss} is equivalent to the original Cauchy problem \eqref{eq:maincauchy} for co--rotational wave maps.
We proceed by a completely abstract approach, i.e., we use operator theoretic methods.

\subsection{Estimates for the nonlinearity}
Our first goal is to formulate the Cauchy problem \eqref{eq:main1stcss} as a nonlinear ODE on the Hilbert space $\mc{H}$.
In order to do so, we need to understand the mapping properties of the nonlinearity.
We are going to need Hardy's inequality in the following form.
\begin{lemma}[Hardy's inequality]
 Let $\alpha>1$. Then
 $$ \int_0^1 \frac{|u(\rho)|^2}{\rho^\alpha}d\rho \leq \left (\frac{2}{\alpha-1} \right)^2 
\int_0^1 \frac{|u'(\rho)|^2}{\rho^{\alpha-2}}d\rho $$
 for all $u \in C_c^\infty(0,1]$.
\end{lemma}

\begin{proof}
 Integration by parts and Cauchy--Schwarz yield
 \begin{align*}
  \int_0^1 \frac{|u(\rho)|^2}{\rho^\alpha}d\rho&=
\left . -\frac{1}{\alpha-1}\frac{|u(\rho)|^2}{\rho^{\alpha-1}} \right |_0^1+\frac{1}{\alpha-1}\int_0^1 \frac{u'(\rho)\overline{u(\rho)}+u(\rho)\overline{u'(\rho)}}{\rho^{\alpha-1}}d\rho \\
&=-\frac{1}{\alpha-1}|u(1)|^2 +\frac{2}{\alpha-1}\int_0^1 
\frac{\Re [u'(\rho)\overline{u(\rho)}]}{\rho^{\alpha-1}}d\rho \\
&\leq \frac{2}{\alpha-1}\int_0^1 \frac{|u(\rho)|}{\rho^{\alpha/2}}\frac{|u'(\rho)|}{\rho^{\alpha/2-1}}d\rho \\
&\leq \frac{2}{\alpha-1}\left (\int_0^1 \frac{|u(\rho)|^2}{\rho^{\alpha}}d\rho \right )^{1/2} 
\left ( \int_0^1 \frac{|u'(\rho)|^2}{\rho^{\alpha-2}}d\rho \right )^{1/2}
 \end{align*}
for all $u \in C_c^\infty(0,1]$ which implies the claim.
\end{proof}

We will also frequently use the following general fact on continuous extensions.

\begin{lemma}
 \label{lem:context}
 Let $X,Y$ be Banach spaces with norms $\|\cdot\|_X$, $\|\cdot\|_Y$, respectively.
 Assume further that $\tilde{X} \subset X$ is a dense subset of $X$ and let $\tilde{F}: \tilde{X} \to Y$ be a mapping that satisfies the estimate
 $$ \|\tilde{F}(x_1)-\tilde{F}(x_2)\|_Y \leq \|x_1-x_2\|_X^\alpha \gamma (\|x_1\|_X, \|x_2\|_X) $$
 for all $x_1,x_2 \in \tilde{X}$ where $\gamma: [0,\infty) \times [0,\infty) \to [0,\infty)$ is a continuous function and $\alpha>0$.
 Then there exists a unique continuous mapping $F: X \to Y$ with $F(x)=\tilde{F}(x)$ for all $x \in \tilde{X}$ and
  $$ \|F(x_1)-F(x_2)\|_Y \leq \|x_1-x_2\|_X^\alpha \gamma (\|x_1\|_X, \|x_2\|_X) $$
  for all $x_1,x_2 \in X$.
\end{lemma}

\begin{proof}
Let $x \in X$.
Then, by the assumed density of $\tilde{X}$, there exists a sequence $(x_j) \subset \tilde{X}$ with $\|x-x_j\|_X \to 0$ as $j \to \infty$.
This implies $\|x_j\|_X \to \|x\|_X$ and from
$$ \|\tilde{F}(x_j)-\tilde{F}(x_k)\|_Y \leq \|x_j-x_k\|_X^\alpha \gamma(\|x_j\|_X, \|x_k\|_X) \to 0 $$
as $j,k \to \infty$ 
we see that $(\tilde{F}(x_j))\subset Y$ is a Cauchy sequence.
Since $Y$ is a Banach space, this Cauchy sequence has a limit 
$$ y:=\lim_{j\to\infty}\tilde{F}(x_j) $$
and we define the mapping $F: X \to Y$ by $F(x):=y$.
For this to be well--defined we have to ensure that $y$ is independent of the chosen sequence $x_j \to x$.
In order to see this, let $(\tilde{x}_j) \subset \tilde{X}$ be another sequence with $\|x-\tilde{x}_j\|_X \to 0$ as $j \to \infty$.
Then we have
$$ \|x_j-\tilde{x}_j\|_X \leq \|x_j-x\|_X+\|x-\tilde{x}_j\|_X \to 0 $$
as $j \to \infty$ and thus,
\begin{align*} 
\|y-\tilde{F}(\tilde{x}_j)\|_Y &\leq \|y-\tilde{F}(x_j)\|_Y+\|\tilde{F}(x_j)-\tilde{F}(\tilde{x}_j)\|_Y \\
& \leq \|y-\tilde{F}(x_j)\|_Y+\|x_j-\tilde{x}_j\|_X^\alpha \gamma(\|x_j\|_X, \|\tilde{x}_j\|_X) \to 0 
\end{align*}
as $j \to \infty$ and this shows that $\lim_{j \to \infty}\tilde{F}(\tilde{x}_j)=y$.
Consequently, $F: X \to Y$ is well--defined.
If $x \in \tilde{X}$ we have $F(x)=\tilde{F}(x)$ by definition.
This shows that $F$ extends $\tilde{F}$ to all of $X$.
Now let $\varepsilon \in (0,1)$ be arbitrary and choose $x_0,x\in X$ with $\|x_0-x\|_X \leq \varepsilon$.
By construction of $F$, we can find $\tilde{x}_0,\tilde{x} \in \tilde{X}$ with $\|x_0-\tilde{x}_0\|_X \leq \varepsilon$, $\|x-\tilde{x}\|_X \leq \varepsilon$ such that $\|F(x_0)-\tilde{F}(\tilde{x}_0)\|_Y \leq \varepsilon$ and $\|F(x)-\tilde{F}(\tilde{x})\|_Y \leq \varepsilon$.
This implies
$$ \|\tilde{x}_0-\tilde{x}\|_X\leq \|\tilde{x}_0-x_0\|_X+\|x_0-x\|_X+\|x-\tilde{x}\|_X \leq 3 \varepsilon $$
and we conclude 
\begin{align*}
\|F(x_0)-F(x)\|_Y&\leq \|F(x_0)-\tilde{F}(\tilde{x}_0)\|_Y+\|\tilde{F}(\tilde{x}_0)-\tilde{F}(\tilde{x})\|_Y
+\|\tilde{F}(\tilde{x})-F(x)\|_Y \\
&\leq 2\varepsilon +\|\tilde{x}-\tilde{x}_0\|_X^\alpha \gamma(\|\tilde{x}\|_X, \|\tilde{x}_0\|_X)
\leq 2\varepsilon + (3\varepsilon)^\alpha \sup_{0 \leq s, t \leq \|x_0\|_X+4} \gamma(s,t)
\end{align*}
since 
\begin{align*} \|\tilde{x}\|_X &\leq \|x_0\|_X+\|\tilde{x}-x_0\|_X\leq \|x_0\|_X+\|\tilde{x}-\tilde{x}_0\|_X+\|\tilde{x}_0-x_0\|_X \\
&\leq \|x_0\|_X+4\varepsilon \leq \|x_0\|_X+4
\end{align*}
and
$$ \|\tilde{x}_0\|_X\leq \|x_0\|_X+\|\tilde{x}_0-x_0\|_X \leq \|x_0\|_X+\varepsilon \leq \|x_0\|_X+4 $$
for any $\varepsilon \in (0,1)$.
Consequently, we obtain
$$ \|F(x_0)-F(x)\|_Y \leq 2\varepsilon + C_{x_0}\varepsilon^\alpha $$
for a constant $C_{x_0}>0$ depending on $\|x_0\|_X$.
This constant is finite since the continuous function $\gamma$ attains its maximum on the compact set $[0,\|x_0\|_X+4] \times [0,\|x_0\|_X+4]$.
This shows that $F: X \to Y$ is continuous.
Uniqueness of $F$ follows from continuity and the fact that $F$ and $\tilde{F}$ coincide on a dense subset.
Finally, the claimed estimate for $F$ follows from the estimate for $\tilde{F}$ which carries over to the extension by continuity.
\end{proof}

We start the analysis of the nonlinearity by defining an auxiliary operator $A$ on $C_c^\infty(0,1]$ by setting
$$ (Au)(\rho):=\frac{1}{\rho^2}\int_0^\rho \rho' u(\rho')d\rho'. $$
Observe that the operator $A$ appears in the argument of the nonlinearity $N_T$ in \eqref{eq:main1stcss}.

\begin{lemma}
\label{lem:estA1}
 The operator $A$ satisfies the estimate
 $$ |Au(\rho)|\lesssim \sqrt{\rho}\|u'\|_{L^2(0,1)} $$
 for all $\rho \in (0,1)$ and all $u \in C_c^\infty(0,1]$.
\end{lemma}

\begin{proof}
 By Cauchy--Schwarz and Hardy's inequality we obtain
\begin{align*}
 |Au(\rho)|&\leq \frac{1}{\rho^2}\int_0^\rho \rho'^2 \frac{|u(\rho')|}{\rho'}d\rho' 
 \leq \frac{1}{\rho^2} \left ( \int_0^\rho \rho'^4 d\rho' \right )^{1/2} \left (\int_0^\rho \frac{|u(\rho')|^2}{\rho'^2}d\rho' \right )^{1/2} \\
 &\lesssim \sqrt{\rho}\left ( \int_0^1 |u'(\rho)|^2 d\rho \right )^{1/2}
\end{align*}
for all $u \in C_c^\infty(0,1]$.
\end{proof}

Another useful observation is the following.

\begin{lemma}
 \label{lem:estA1a}
 The operator $A$ satisfies the estimates
 $$ \int_0^1 \frac{|Au(\rho)|^2}{\rho^2}d\rho \lesssim \int_0^1 |(Au)'(\rho)|^2 d\rho \lesssim 
\|u'\|_{L^2(0,1)}^2 $$
for all $u \in C^\infty_c(0,1]$.
\end{lemma}

\begin{proof}
Note that $u \in C^\infty_c(0,1]$ implies $Au \in C^\infty_c(0,1]$.
Consequently, Hardy's inequality yields
\begin{align*}
 \int_0^1 \frac{|Au(\rho)|^2}{\rho^2}d\rho &\lesssim \int_0^1 |(Au)'(\rho)|^2 d\rho \\
&\lesssim 
\int_0^1 \frac{1}{\rho^6} \left |\int_0^\rho \rho' u(\rho')d\rho' \right |^2 d\rho+\int_0^1 
\frac{|u(\rho)|^2}{\rho^2} d\rho \\
&\lesssim \int_0^1 \frac{|u(\rho)|^2}{\rho^2}d\rho \\
&\lesssim \|u'\|_{L^2(0,1)}^2 
\end{align*}
for all $u \in C^\infty_c(0,1]$.
\end{proof}

From now on we restrict ourselves to real--valued functions.
For the following results recall that $\mc{H}=H_1 \times H_2$ where the respective norms $\|\cdot\|_1$, $\|\cdot\|_2$ on $H_1$, $H_2$ are given by
$$ \|u\|_1^2=\int_0^1 \frac{|u'(\rho)|^2}{\rho^2} d\rho, \quad \|u\|_2^2=\int_0^1 |u'(\rho)|^2 d\rho=\|u'\|_{L^2(0,1)}^2, $$
see Sec.~\ref{sec:oplin} and \cite{DSA}.
Note that the nonlinearity $N_T$ in \eqref{eq:main1stcss} appears in the first component and takes an argument from the second component. Thus, it is supposed to map from $H_2$ to $H_1$.
However, the nonlinearity is composed with the operator $A$ and Lemma \ref{lem:estA1} suggests to include a weighted $L^\infty$--piece in the norm.
Consequently, we view the nonlinearity as a mapping $X \to H_1$ where the Banach space $X$ is defined as the completion of $C^\infty_c(0,1]$ with respect to the norm
$$ \|u\|_X:=\|u\|_2+\sup_{\rho \in (0,1)}\left |\frac{u(\rho)}{\sqrt{\rho}} \right |. $$
Note, however, that this is just for notational convenience since, in fact, by Cauchy--Schwarz, we have
$$ \frac{|u(\rho)|}{\sqrt{\rho}}\leq \frac{1}{\sqrt{\rho}}\int_0^\rho |u'(\rho)|d\rho \leq \|u'\|_{L^2(0,1)} $$
for all $u \in C_c^\infty(0,1]$ and thus, the norms $\|\cdot\|_2$ and $\|\cdot\|_X$ are equivalent.
According to Lemmas \ref{lem:estA1} and \ref{lem:estA1a}, the operator $A$ extends to a bounded linear operator $A: H_2 \to X$.

For a suitable function $F$ of two variables, we define the \emph{composition} or \emph{Nemitsky operator} $\hat{F}$, acting on $C^\infty_c(0,1]$, by 
$$ \hat{F}(u)(\rho):=F(u(\rho),\rho). $$
The following two results give sufficient conditions on $F$ such that the corresponding Nemitsky operator is continuous from $X$ to $H_1$.
 
\begin{lemma}
\label{lem:Nemcont}
  Let $F: \mathbb{R} \times [0,1] \to \mathbb{R}$ be twice continuously differentiable in both variables
and assume that $\partial_1 F(0,\rho)=0$ for all $\rho \in [0,1]$.
  Suppose further that there exists a constant $C>0$ such that
  $$ |\partial_{1j}F(x,\rho)|\leq C(|x|+\rho) $$
  for all $(x,\rho) \in \mathbb{R}\times [0,1]$ and $j=1,2$.
  Then we have $\hat{F}(u) \in H_1$ for any $u \in C^\infty_c(0,1]$ and there exists a continuous function $\gamma: [0,\infty) \times [0,\infty) \to [0,\infty)$ such that $\hat{F}$ satisfies the estimate
  $$ \|\hat{F}(u)-\hat{F}(v)\|_1 \leq \|u-v\|_X \gamma(\|u\|_X, \|v\|_X) $$
  for all $u,v \in C^\infty_c(0,1]$.
  As a consequence, $\hat{F}$ uniquely extends to a continuous map $\hat{F}: X \to H_1$ and the above estimate remains valid for all $u,v \in X$.
\end{lemma}

\begin{proof}
 Let $u,v \in C^\infty_c(0,1]$.
 By definition of $\hat{F}$ and $\|\cdot\|_1$, we have
 \begin{align}
 \label{eq:proofNemcont}
  \|\hat{F}(u)-\hat{F}(v)\|_1^2 &\lesssim \int_0^1 \frac{|\partial_1 F(u(\rho),\rho)u'(\rho)-\partial_1 F(v(\rho),\rho)v'(\rho)|^2}{\rho^2}d\rho \\
  &\quad + \int_0^1 \frac{|\partial_2 F(u(\rho),\rho)
-\partial_2 F(v(\rho),\rho)|^2}{\rho^2}d\rho \nonumber \\
  &\lesssim \int_0^1 \frac{\left |\partial_1 F(u(\rho),\rho)
\left [u'(\rho)-v'(\rho) \right ] \right |^2}{\rho^2}d\rho \nonumber \\
&\quad +\int_0^1 \frac{\left | \left [\partial_1 F(u(\rho),\rho)-\partial_1 F(v(\rho),\rho) \right ] v'(\rho) \right |^2}{\rho^2}d\rho \nonumber \\
&\quad + \int_0^1 \frac{|\partial_2 F(u(\rho),\rho)
-\partial_2 F(v(\rho),\rho)|^2}{\rho^2}d\rho=:A_1+A_2+A_3 \nonumber 
 \end{align}
 Now we put the terms containing $u',v'$ into $L^2$ and the rest into $L^\infty$, i.e.,
 $$ A_1\leq \sup_{\rho \in (0,1)}\left | \frac{\partial_1 F(u(\rho),\rho)}{\rho} \right |^2 \int_0^1 |u'(\rho)-v'(\rho)|^2 d\rho $$
 and, analogously,
 \begin{align*} 
A_2 &\leq \sup_{\rho \in (0,1)}\left | \frac{\partial_1 F(u(\rho),\rho)-\partial_1 F(v(\rho),\rho)}{\rho} \right |^2 \int_0^1 |v'(\rho)|^2 d\rho \\
A_3 &\leq \sup_{\rho \in (0,1)}\left | \frac{\partial_2 F(u(\rho),\rho)
-\partial_2 F(v(\rho),\rho)}{\rho} \right |^2.
 \end{align*}
Using the assumption on $\partial_{1j}F$ and applying the fundamental theorem of calculus we obtain
\begin{align*} \left | \partial_{j}F(x,\rho)-\partial_{j}F(y,\rho) \right | 
&=\left |\int_0^1 \partial_{j1}F(y+t(x-y),\rho)(x-y)dt \right | \\
&\leq C|x-y| \int_0^1 \left (|y+t(x-y)|+\rho \right )dt \\
&\lesssim |x-y| \left ( |x|+|y|+\rho \right )
\end{align*}
 for all $x,y \in \mathbb{R}$, all $\rho \in [0,1]$ and $j=1,2$.
In particular, this implies
\begin{equation} 
\label{eq:proofNemest}
|\partial_1 F(x,\rho)|\lesssim |x| \left (|x|+\rho \right )  
\end{equation}
since $\partial_1 F(0,\rho)=0$ for all $\rho \in [0,1]$ by assumption.
Consequently, we obtain
\begin{align*} 
A_1 &\lesssim \sup_{\rho \in (0,1)}\left | \frac{|u(\rho)|\left (|u(\rho)|+
\rho \right )}{\rho} \right |^2 \|u-v\|_2^2 \\
&\lesssim \sup_{\rho \in (0,1)}\left |\frac{u(\rho)}{\sqrt{\rho}} \right |^2 \sup_{\rho \in (0,1)} \left | \frac{|u(\rho)|+\rho}{\sqrt{\rho}} \right |^2 \|u-v\|_2^2 \\
&\lesssim \|u\|_X^2 (\|u\|_X^2+1)\|u-v\|_X^2
\end{align*}
as well as
\begin{align*}
 A_2 & \lesssim \sup_{\rho \in (0,1)}\left | \frac{|u(\rho)-v(\rho)|(|u(\rho)|+|v(\rho)|+\rho)}{\rho} \right |^2 \|v\|_2^2 \\
 &\lesssim \sup_{\rho \in (0,1)}\left |\frac{|u(\rho)-v(\rho)|}{\sqrt{\rho}} \right |^2
 \sup_{\rho \in (0,1)}\left |\frac{|u(\rho)|+|v(\rho)|+\rho}{\sqrt{\rho}} \right |^2 \|v\|_2^2 \\
 &\lesssim \|u-v\|_X^2 \left (\|u\|_X^2+\|v\|_X^2+1 \right )\|v\|_X^2
\end{align*}
and, analogously,
$$ A_3 \lesssim \|u-v\|_X^2 \left (\|u\|_X^2+\|v\|_X^2+1 \right ). $$
Putting everything together we arrive at the claimed estimate.
The existence of the extension with the stated properties follows from Lemma \ref{lem:context}.
\end{proof}

Next, we turn to the question of differentiability of the Nemitsky operator $\hat{F}$.
Recall that the G\^ ateaux derivative $D_G\hat{F}(u): X \to H_1$ of $\hat{F}$ at $u \in X$ is defined as 
$$ D_G \hat{F}(u)v:=\lim_{h \to 0}\frac{\hat{F}(u+hv)-\hat{F}(u)}{h} $$
provided the right--hand side exists and defines a bounded linear operator from $X$ to $H_1$.
We give sufficient conditions for the G\^ateaux derivative to exist.
 
\begin{lemma}
\label{lem:NemGdiff}
 Let $F \in C^3(\mathbb{R} \times [0,1])$ satisfy the assumptions from Lemma \ref{lem:Nemcont}.
 Assume further that there exists a constant $C>0$ such that
 $$ |\partial_{11j} F(x,\rho)|\leq C $$
 for all $(x,\rho) \in \mathbb{R} \times [0,1]$ and $j=1,2$.
 Then, at every $u \in C^\infty_c(0,1]$, the Nemitsky operator $\hat{F}: X \to H_1$ 
is G\^ateaux differentiable and its G\^ateaux derivative at $u$ applied to $v \in X$ is given by
 $$ [D_G \hat{F}(u)v](\rho)=\partial_1 F(u(\rho),\rho)v(\rho). $$
\end{lemma}

\begin{proof}
 According to Lemma \ref{lem:Nemcont}, $\hat{F}$ extends to a continuous operator $\hat{F}: X \to H_1$.
 For $u,v \in C^\infty_c(0,1]$ define
 $$ [\tilde{D}_G \hat{F}(u)v](\rho):=\partial_1 F(u(\rho),\rho)v(\rho). $$
 Inserting the definition of $\|\cdot\|_1$ yields
 \begin{align*}
  \|\tilde{D}_G \hat{F}(u)v\|_1^2&\lesssim \int_0^1 \frac{|\partial_{11}F(u(\rho),\rho)u'(\rho)v(\rho)|^2}{\rho^2}d\rho 
  +\int_0^1 \frac{|\partial_{12}F(u(\rho),\rho)v(\rho)|^2}{\rho^2}d\rho \\
  &\quad +\int_0^1 \frac{|\partial_1 F(u(\rho),\rho)v'(\rho)|^2}{\rho^2}d\rho \\
  &\lesssim \sup_{\rho \in (0,1)}\left |\frac{\partial_{11} F(u(\rho),\rho)v(\rho)}{\rho} \right |^2
  \int_0^1 |u'(\rho)|^2 d\rho
  +\sup_{\rho \in (0,1)}\left | \frac{\partial_{12}F(u(\rho),\rho)v(\rho)}{\rho} \right |^2 \\
  &\quad +\sup_{\rho \in (0,1)}\left |\frac{\partial_1 F(u(\rho),\rho)}{\rho} \right |^2 \int_0^1 |v'(\rho)|^2 d\rho
 \end{align*}
and by assumption we have
$$ \sup_{\rho \in (0,1)} \left |\frac{\partial_{1j}F(u(\rho),\rho)v(\rho)}{\rho} \right |^2 \lesssim 
 \sup_{\rho \in (0,1)}\left |\frac{(|u(\rho)|+\rho)|v(\rho)|}{\rho} \right |^2
 \lesssim \left (\|u\|_X^2+1 \right )\|v\|_X^2 $$
 for $j=1,2$.
 Furthermore, Eq.~\eqref{eq:proofNemest} yields
 $$ \sup_{\rho \in (0,1)}\left |\frac{\partial_1 F(u(\rho),\rho)}{\rho} \right |^2 \lesssim 
\sup_{\rho \in (0,1)}\left |\frac{|u(\rho)|(|u(\rho)|+\rho)}{\rho} \right |^2 
\lesssim \|u\|_X^2 \left (\|u\|_X^2+1 \right ) $$
and we infer
$$ \|\tilde{D}_G \hat{F}(u)v\|_1 \leq C_u \|v\|_X $$
for all $v \in C^\infty_c(0,1]$ where $C_u>0$ is a constant that depends on $\|u\|_X$.
Since $C^\infty_c(0,1]$ is dense in $X$, $\tilde{D}_G \hat{F}(u)$ extends to a bounded linear operator from $X$ to $H_1$.
Recall that, for general $v \in X$, the extension is defined by 
$$\tilde{D}_G \hat{F}(u)v=\lim_{j \to \infty}\tilde{D}_G \hat{F}(u)v_j $$
where $(v_j) \subset C^\infty_c(0,1]$ is an arbitrary sequence with $v_j \to v$ in $X$ and the limit is taken in $H_1$.
Since convergence in $H_1$ or $X$ implies pointwise convergence (see Lemma \ref{lem:Hcont}), the representation
$$ [\tilde{D}_G \hat{F}(u)v](\rho)=\partial_1 F(u(\rho),\rho)v(\rho) $$
remains valid for general $v \in X$.
We claim that $\tilde{D}_G \hat{F}(u)$ is the G\^ateaux derivative of $\hat{F}$ at $u$.
In order to prove this, we have to show that
$$ \lim_{h \to 0} \left \|\frac{\hat{F}(u+hv)-\hat{F}(u)}{h}-\tilde{D}_G \hat{F}(u)v \right \|_1=0 $$
for any $v \in X$.

Assume for the moment that $v \in C^\infty_c(0,1]$.
By applying the fundamental theorem of calculus we obtain
$$ \hat{F}(u+hv)(\rho)-\hat{F}(u)(\rho)=F(u(\rho)+hv(\rho), \rho)-F(u(\rho),\rho)
=\int_0^h \partial_1 F(u(\rho)+\tilde{h}v(\rho),\rho)v(\rho)d\tilde{h} $$
and this implies
\begin{align*} 
g_h(\rho)&:=\frac{\hat{F}(u+hv)(\rho)-\hat{F}(u)(\rho)}{h}-[\tilde{D}_G \hat{F}(u)v](\rho) \\
&=\frac{v(\rho)}{h}\int_0^h \left [ \partial_1 F(u(\rho)+\tilde{h}v(\rho),\rho)-\partial_1 F(u(\rho),\rho) \right ] d\tilde{h}.
\end{align*}
Now observe that
\begin{align*} &\frac{d}{d\rho}\int_0^h \left [ \partial_1 F(u(\rho)+\tilde{h}v(\rho),\rho)-\partial_1 F(u(\rho),\rho) \right ] d\tilde{h} \\
&=\int_0^h \partial_\rho \left [ \partial_1 F(u(\rho)+\tilde{h}v(\rho),\rho)-\partial_1 F(u(\rho),\rho) \right ] d\tilde{h} 
\end{align*}
for all $h \in [0,h_0]$
by dominated convergence since the function 
$$ (\tilde{h},\rho)\mapsto \partial_1 F(u(\rho)+\tilde{h}v(\rho),\rho)-\partial_1 F(u(\rho),\rho) $$
belongs to $C^1([0,h_0] \times [0,1])$ for some $h_0>0$.
Consequently, we infer
\begin{align*}
|g_h'(\rho)|&\leq |v'(\rho)|\sup_{\tilde{h} \in (0,h)}\left | 
\partial_1 F(u(\rho)+\tilde{h}v(\rho), \rho)-\partial_1 F(u(\rho),\rho) \right | \\
&\quad +|v(\rho)||u'(\rho)|\sup_{\tilde{h} \in (0,h)}\left |\partial_{11}F(u(\rho)+\tilde{h}v(\rho),\rho)-\partial_{11} F(u(\rho),\rho) \right | \\
&\quad + |v(\rho)|\sup_{\tilde{h} \in (0,h)}\left | \partial_{12}F(u(\rho)+\tilde{h}v(\rho),\rho)-\partial_{12}F(u(\rho),\rho) \right | \\
&\quad + |v'(\rho)||v(\rho)|\sup_{\tilde{h} \in (0,h)} \left |\tilde{h} \partial_{11}F(u(\rho)+\tilde{h}v(\rho),\rho) \right |.
\end{align*}
Note that the partial derivatives of $F$ commute since $F$ is three--times continuously differentiable. 
Furthermore, by the assumption on $\partial_{11j}F$, we have
\begin{align*} 
&\sup_{\tilde{h} \in (0,h)} \left |
\partial_{j1}F(x+\tilde{h}y,\rho)-\partial_{j1} F(x,\rho) \right | \\
&=\sup_{\tilde{h} \in (0,h)} \left | \int_0^{\tilde{h}} \partial_{j11}
F(x+h_1 y,\rho)y\:dh_1 \right | \\
&\leq C h |y|
\end{align*}
for $j=1,2$ and, from the proof of Lemma \ref{lem:Nemcont},
\begin{align*} 
\sup_{\tilde{h} \in (0,h)}\left |\partial_1 F(x+\tilde{h}y,\rho)-\partial_1 F(x,\rho) \right | 
&\leq \sup_{\tilde{h} \in (0,h)} \tilde{h}|y|\left (|x|+|x+\tilde{h}y|+\rho \right ) \\
&\lesssim h |y| \left (|x|+|y|+\rho \right ) 
\end{align*}
for all $x,y \in \mathbb{R}$, all $\rho \in [0,1]$ and $h \in [0,h_0]$.
Also, by the assumption on $\partial_{11}F$, we have
$$ \sup_{\tilde{h}\in (0,h)} \left |\tilde{h} \partial_{11}F(x+\tilde{h}y,\rho) \right |
\lesssim \sup_{\tilde{h}\in (0,h)} \tilde{h}\left (|x+\tilde{h}y|+\rho \right )\leq h \left (|x|+|y|
+\rho \right ) $$
in the above stated domains for $x,y,\rho$ and $h$.
Consequently, we obtain
\begin{align*}
 \|g_h\|_1^2&=\int_0^1 \frac{|g_h'(\rho)|^2}{\rho^2}d\rho \\
 &\lesssim h^2 \sup_{\rho \in (0,1)}\left |\frac{|v(\rho)|\left (|u(\rho)|+|v(\rho)|+\rho \right )}{\rho} \right |^2 \int_0^1 |v'(\rho)|^2 d\rho 
 + h^2 \sup_{\rho \in (0,1)} \left | \frac{|v(\rho)|^2}{\rho} \right |^2 \int_0^1 |u'(\rho)|^2 d\rho \\
 &\quad +h^2 \sup_{\rho \in (0,1)}\left | \frac{|v(\rho)|^2}{\rho} \right |^2
 +h^2 \sup_{\rho \in (0,1)} \left |\frac{|v(\rho)|\left (|u(\rho)|+|v(\rho)|+\rho\right )}{\rho} \right |^2 \int_0^1 |v'(\rho)|^2 d\rho
\end{align*}
and this shows that
$\|g_h\|_1\leq h\gamma(\|u\|_X, \|v\|_X)$
for a suitable continuous function $\gamma: [0,\infty) \times [0,\infty) \to [0,\infty)$.
By Lemma \ref{lem:Nemcont}, we know that $\hat{F}$ extends to a continuous map from $X$ to $H_1$.
Now let $v \in X$ be arbitrary and choose a sequence $(v_j) \subset C^\infty_c(0,1]$ with $v_j \to v$ in $X$.
By the continuity of $\hat{F}$, this implies $\hat{F}(u+hv_j) \to \hat{F}(u+hv)$ in $H_1$ for any $u \in C^\infty(0,1]$ and $h \in (0,h_0)$.
We conclude that
\begin{align*} 
\left \|\frac{\hat{F}(u+hv)-\hat{F}(u)}{h}-\tilde{D}_G \hat{F}(u)v \right \|_1 &=
\lim_{j \to \infty} \left \|\frac{\hat{F}(u+hv_j)-\hat{F}(u)}{h}-\tilde{D}_G \hat{F}(u)v_j \right \|_1 \\
&\leq h\lim_{j\to \infty}\gamma(\|u\|_X, \|v_j\|_X)=h\gamma(\|u\|_X, \|v\|_X)
\end{align*}
by the continuity of $\tilde{D}_G \hat{F}(u)$ and $\gamma$.
This shows that, for any $v \in X$, we have
\begin{equation} 
\label{eq:proofNemGdiff}
\left \|\frac{\hat{F}(u+hv)-\hat{F}(u)}{h}-\tilde{D}_G \hat{F}(u)v \right \|_1 \leq h\gamma(\|u\|_X, \|v\|_X) \to 0 
\end{equation}
as $h \to 0$
and therefore, $\tilde{D}_G\hat{F}(u)=D_G \hat{F}(u)$, the G\^ateaux derivative of $\hat{F}$ at $u \in C^\infty_c(0,1]$. 
\end{proof}

Finally, we show that the G\^ateaux derivative is in fact a Fr\'echet derivative.

\begin{lemma}
 \label{lem:NemFdiff}
 Let $F$ satisfy the assumptions from Lemma \ref{lem:NemGdiff}.
 Then $\hat{F}: X \to H_1$ is G\^ateaux differentiable at every $u \in X$ and the G\^ateaux derivative $D_G \hat{F}: X \to \mc{B}(X,H_1)$ is continuous and satisfies the estimate
$$ \|D_G \hat{F}(u)-D_G \hat{F}(\tilde{u})\|_{\mc{B}(X,H_1)}\leq \|u-\tilde{u}\|_X\; \gamma(\|u\|_X, \|\tilde{u}\|_X ) $$
for all $u, \tilde{u} \in X$ where $\gamma: [0,\infty) \times [0,\infty) \to [0,\infty)$ is a suitable continuous function.
 As a consequence, $\hat{F}$ is continuously Fr\'echet differentiable and $D\hat{F}=D_G \hat{F}$.
\end{lemma}

\begin{proof}
 Let $u,\tilde{u},v \in C^\infty_c(0,1]$. 
 According to Lemma \ref{lem:NemGdiff}, $\hat{F}$ is G\^ateaux differentiable at $u$ and $\tilde{u}$ 
and we set
 $$ g_{u,\tilde{u},v}(\rho):=[D_G \hat{F}(u)v](\rho)-[D_G \hat{F}(\tilde{u})v](\rho)
 =\left [\partial_1 F(u(\rho),\rho)-\partial_1 F(\tilde{u}(\rho), \rho) \right ]v(\rho)$$
 where we have used Lemma \ref{lem:NemGdiff} again.
 Differentiating with respect to $\rho$ we obtain
 \begin{align*}
  |g_{u,\tilde{u},v}'(\rho)|&\leq |v(\rho)| \left |\partial_{11}F(u(\rho),\rho)u'(\rho)
-\partial_{11}F(\tilde{u}(\rho),\rho)\tilde{u}'(\rho)\right |\\
&\quad +|v(\rho)| \left | \partial_{12}F(u(\rho),\rho)-\partial_{12}F(\tilde{u}(\rho),\rho) \right |\\
&\quad +|v'(\rho)| \left |\partial_1 F(u(\rho),\rho)-\partial_1 F(\tilde{u}(\rho),\rho) \right |.
 \end{align*}
 By the assumptions on $F$ we have
 \begin{align*} 
 |\partial_{j1}F(x,\rho)-\partial_{j1}F(\tilde{x},\rho)|\leq |x-\tilde{x}|\int_0^1 |\partial_{j11}F(\tilde{x}+t(x-\tilde{x}), \rho)|dt \lesssim |x-\tilde{x}| 
 \end{align*}
for $j=1,2$ as well as
$$ |\partial_1 F(x,\rho)-\partial_1 F(\tilde{x},\rho)|\lesssim |x-\tilde{x}|\left (|x|+|\tilde{x}|
+\rho \right ) $$
for all $x, \tilde{x} \in \mathbb{R}$ and $\rho \in [0,1]$ (cf.~the proofs of Lemmas \ref{lem:Nemcont} and \ref{lem:NemGdiff}).
This shows that 
\begin{align*}
\left |\partial_{11}F(u(\rho),\rho)u'(\rho)-\partial_{11} F(\tilde{u}(\rho),\rho)\tilde{u}'(\rho) \right | &\leq
|\partial_{11}F(u(\rho),\rho)| |u'(\rho)-\tilde{u}'(\rho)| \\
&\quad + |\partial_{11}F(u(\rho),\rho)-\partial_{11}F(\tilde{u}(\rho),\rho)||\tilde{u}'(\rho)| \\
&\lesssim \left | |u(\rho)|+\rho \right | |u'(\rho)-\tilde{u}'(\rho)|+|u(\rho)-\tilde{u}(\rho)||\tilde{u}'(\rho)|
\end{align*}
and we obtain
\begin{align*}
 \|g_{u,\tilde{u},v}(\rho)\|_1^2&=\int_0^1 \frac{|g_{u,\tilde{u},v}'(\rho)|^2}{\rho^2}d\rho
 \lesssim \sup_{\rho \in (0,1)}\left | \frac{|v(\rho)|(|u(\rho)|+\rho)}{\rho} \right |^2 \int_0^1 |u'(\rho)-\tilde{u}'(\rho)|^2 d\rho \\
 &\quad + \sup_{\rho \in (0,1)} \left | \frac{|v(\rho)||u(\rho)-\tilde{u}(\rho)|}{\rho} \right |^2
 \int_0^1 |\tilde{u}'(\rho)|^2 d\rho +\sup_{\rho \in (0,1)} \left |\frac{|v(\rho)||u(\rho)-\tilde{u}(\rho)|}{\rho} \right |^2 \\
 &\quad + \sup_{\rho \in (0,1)}\left |\frac{|u(\rho)-\tilde{u}(\rho)|(|u(\rho)|+|\tilde{u}(\rho)|+\rho)}{\rho} \right |^2 \int_0^1 |v'(\rho)|^2 d\rho \\
 &\lesssim \|u-\tilde{u}\|_X^2 \gamma^2 (\|u\|_X, \|\tilde{u}\|_X, \|v\|_X)
\end{align*}
for a continuous function $\gamma: [0,\infty)^3 \to [0,\infty)$.
According to Lemma \ref{lem:NemGdiff}, the operator $D_G \hat{F}(u): X \to H_1$ is bounded and thus,
the above estimate extends to all $v \in X$ by approximation.
Consequently, we obtain
\begin{align} 
\label{eq:proofNemFdiff}
\|D_G \hat{F}(u)-D_G\hat{F}(\tilde{u})\|_{\mc{B}(X,H_1)} 
&\lesssim \sup\left \{ \|u-\tilde{u}\|_X\:\gamma(\|u\|_X,\|\tilde{u}\|_X,\|v\|_X): v \in X, \|v\|_X=1 \right \} \\
&=\|u-\tilde{u}\|_X\; \gamma(\|u\|_X,\|\tilde{u}\|_X,1). \nonumber
\end{align}
By Lemma \ref{lem:context}, the operator $D_G \hat{F}$ extends to a continuous map $D_G \hat{F}: X \to \mc{B}(X,H_1)$ and the estimate \eqref{eq:proofNemFdiff} remains valid for all $u, \tilde{u} \in X$.
Furthermore, Eq.~\eqref{eq:proofNemGdiff} shows that, for arbitrary $u \in X$, $D_G \hat{F}(u)$ is the G\^ateaux derivative of $\hat{F}$ at $u$.
Since $D_G \hat{F}$ is continuous, it follows from \cite{zeidler}, p.~137, Proposition 4.8 that $\hat{F}$ is Fr\'echet differentiable at every $u \in X$ and $D_G \hat{F}=D\hat{F}$.
\end{proof}

Now it is time to recall the definition \eqref{eq:nonlinearity} of the nonlinearity which reads
$$ N_T(u)=\sin(2(\psi^T+u))-\sin(2\psi^T)-2\cos(2\psi^T)u. $$
However, due to our definition of the similarity coordinates we have
$$ \psi^T(t,r)=2 \arctan \left (\tfrac{r}{T-t} \right )=f_0(\rho) $$
and thus, the nonlinearity does in fact not depend on $T$.
Consequently, we drop the subscript $T$ in the sequel and write $N$ instead of $N_T$.

\begin{lemma}
 \label{lem:NFdiff}
 The Nemitsky operator $\hat{N}$ associated to the nonlinearity $N$ extends to a continuous map $\hat{N}: X \to H_1$. Furthermore, $\hat{N}$ is continuously Fr\'echet differentiable at every $u \in X$ and the Fr\'echet derivative $D\hat{N}$ satisfies the estimate
 $$ \|D\hat{N}(u)-D\hat{N}(\tilde{u})\|_{\mc{B}(X,H_1)} \leq \|u-\tilde{u}\|_X\; \gamma(\|u\|_X, \|\tilde{u}\|_X) $$
 for all $u, \tilde{u} \in X$ and a suitable continuous function $\gamma: [0,\infty) \times [0,\infty) \to [0,\infty)$. 
\end{lemma}

\begin{proof}
We have 
$$ N(x,\rho)=\sin(2f_0(\rho)+2x)-\sin(2f_0(\rho))-2\cos(2f_0(\rho))x $$
 and it suffices to verify the assumptions of Lemma \ref{lem:NemFdiff} for $N$.
 Obviously, we have $N \in C^3(\mathbb{R}\times [0,1])$ and
 $$ \partial_1 N(x,\rho)=2\cos (2f_0(\rho)+2x)-2\cos(2f_0(\rho)) $$
 shows that $\partial_1 N(0,\rho)=0$ for all $\rho \in [0,1]$.
 Furthermore,
 \begin{align*}
\partial_{11}N(x,\rho)&=-4 \sin(2f_0(\rho)+2x) \\
\partial_{12}N(x,\rho)&=-4\sin(2f_0(\rho)+2x)f_0'(\rho)+4\sin(2f_0(\rho))f_0'(\rho) 
 \end{align*}
 and
 \begin{align*}
\partial_{111}N(x,\rho)&=-8\cos(2f_0(\rho)+2x) \\
\partial_{121}N(x,\rho)&=-8\cos(2f_0(\rho)+2x)f_0'(\rho)
 \end{align*}
which shows that there exists a constant $C>0$ such that
$$ |\partial_{11j}N(x,\rho)|\leq C $$
 for $j=1,2$ and all $(x,\rho) \in \mathbb{R} \times [0,1]$.
Consequently, we obtain
$$ |\partial_{1j}N(x,\rho)|\leq |x|\int_0^1 |\partial_{11j}N(tx,\rho)|dt+|\partial_{1j}N(0,\rho)|
 \lesssim |x|+\rho $$
 for $j=1,2$ and all $(x,\rho) \in \mathbb{R} \times [0,1]$
and we conclude that $N$ indeed satisfies the assumptions of Lemma \ref{lem:NemFdiff} which yields the claim. 
\end{proof}

After these preparations we are now ready to define the vector--valued nonlinearity $\mb{N}: \mc{H} \to \mc{H}$ by
$$ \mb{N}(\mb{u}):=\left ( \begin{array}{c} \hat{N}(A u_2) \\ 0 \end{array} \right ). $$
As a simple consequence of Lemma \ref{lem:NFdiff}, $\mb{N}$ is continuously Fr\'echet differentiable.

\begin{lemma}
\label{lem:N}
 The nonlinearity $\mb{N}: \mc{H} \to \mc{H}$ is continuously Fr\'echet differentiable at any $\mb{u} \in \mc{H}$.
 Furthermore, there exist continuous functions $\gamma_1: [0,\infty) \to [0,\infty)$ 
and $\gamma_2: [0,\infty) \times [0,\infty) \to [0,\infty)$ such that 
 $$ \|\mb{N}(\mb{u})\|\leq \|\mb{u}\|^2 \gamma_1(\|\mb{u}\|) $$
 and
 $$ \|D\mb{N}(\mb{u})\mb{v}-D\mb{N}(\tilde{\mb{u}})\mb{v}\|
\leq \|\mb{u}-\tilde{\mb{u}}\|\|\mb{v}\|\gamma_2(\|\mb{u}\|, \|\tilde{\mb{u}}\|) $$
 for all $\mb{u}, \tilde{\mb{u}}, \mb{v} \in \mc{H}$.
 Finally, $D\mb{N}(\mb{0})=\mb{0}$.
\end{lemma}

\begin{proof}
 Define auxiliary operators $B: \mc{H} \to X$ and $\mb{T}: H_1 \to \mc{H}$ by
 $$ B \mb{u}:=A u_2 \mbox{ and } \mb{T} u:=\left ( \begin{array}{c}u \\ 0 \end{array} \right ). $$
 Obviously, $\mb{T}$ is linear and bounded.
 The same is true for $B$ since
 $$ \|B \mb{u}\|_X=\|Au_2 \|_X \lesssim \|u_2\|_2 \lesssim \|\mb{u}\| $$
 by Lemmas \ref{lem:estA1} and \ref{lem:estA1a}.
 Therefore, $\mb{N}$ can be written as $\mb{N}=\mb{T} \circ \hat{N} \circ B$. 
 Consequently, Lemma \ref{lem:NFdiff} and the chain rule show that $\mb{N}$ is Fr\'echet differentiable at every $\mb{u} \in \mc{H}$ and we obtain
 $$ D\mb{N}(\mb{u})=D\mb{T}(\hat{N}(B\mb{u}))D\hat{N}(B\mb{u})DB(\mb{u})=\mb{T}D\hat{N}(B\mb{u})B. $$
 By Lemma \ref{lem:NemGdiff} we have
 $$ [D\hat{N}(0)v](\rho)=\partial_1 N(0,\rho)v(\rho)=0 $$
 which shows $D\mb{N}(\mb{0})=\mb{0}$.
 Furthermore, the estimate from Lemma \ref{lem:NFdiff} yields
 \begin{align*}
  \|D\mb{N}(\mb{u})\mb{v}-D\mb{N}(\tilde{\mb{u}})\mb{v}\|& \lesssim 
  \|D\hat{N}(B\mb{u})B\mb{v}-D\hat{N}(B\tilde{\mb{u}})B\mb{v}\| \\
  &\lesssim \|B\mb{u}-B\tilde{\mb{u}}\|_X \|B\mb{v}\|_X \; \gamma(\|B\mb{u}\|_X, \|B\tilde{\mb{u}}\|_X) \\
  &\lesssim \|\mb{u}-\tilde{\mb{u}}\| \|\mb{v}\| \; \gamma(\|B\mb{u}\|_X, \|B\tilde{\mb{u}}\|_X) 
 \end{align*}
and, with $D\mb{N}(\mb{0})=\mb{0}$, this shows in particular that
$$ \|D\mb{N}(\mb{u})\mb{v}\|\lesssim \|\mb{u}\|\|\mb{v}\|\; \gamma(\|B\mb{u}\|_X, 0) $$
for all $\mb{u}, \mb{v} \in \mc{H}$.
Consequently, by using the fundamental theorem of calculus, we obtain
\begin{align*}
 \|\mb{N}(\mb{u})\|\leq \int_0^1 \|D\mb{N}(t\mb{u})\mb{u}\|dt 
\leq \|\mb{u}\|^2 \int_0^1 \gamma(t\|B\mb{u}\|_X, 0)dt
\end{align*}
since $\mb{N}(\mb{0})=\mb{0}$.
\end{proof}

\subsection{Global existence for the nonlinear problem}
We establish global existence for the wave maps equation \eqref{eq:main1stcss}.
However, at the moment we do not allow for arbitrary initial data but modify the given data along the one--dimensional subspace spanned by the gauge mode. 
Thereby, we suppress the gauge instability.
The main result of this section can be viewed as a nonlinear version of Theorem \ref{thm:linear}.

As a first step we reformulate Eq.~\eqref{eq:main1stcss} as a nonlinear ODE on the Hilbert space $\mc{H}$.
With the nonlinear mapping $\mb{N}: \mc{H} \to \mc{H}$ from above we can simply write
\begin{equation}
\label{eq:maincssop}
 \left \{ \begin{array}{l}
\frac{d}{d \tau}\Phi(\tau)=L\Phi(\tau)+\mb{N}(\Phi(\tau)) \mbox{ for }\tau>-\log T \\
\Phi(-\log T)=\mb{u}
          \end{array} \right .
\end{equation}
for a function $\Phi: [-\log T,\infty) \to \mc{H}$ with initial data $\mb{u}$.
We do not specify the initial data explicitly but keep them general in this section.
Our aim is to construct a mild solution of this equation.
By a mild solution we mean a solution of the associated integral equation
\begin{equation}
\label{eq:maincssopmild}
 \Phi(\tau)=S(\tau+\log T)\mb{u}+\int_{-\log T}^\tau S(\tau-\tau')\mb{N}(\Phi(\tau'))d\tau', \quad \tau \geq -\log T 
\end{equation}
where $S$ is the semigroup from Theorem \ref{thm:linear}.
Here, the integral is well--defined as a Riemann integral over a continuous function with values in a Banach space.
Consequently, we restrict ourselves to continuous solutions 
$\Phi: [-\log T,\infty) \to \mc{H}$.
Any solution of Eq.~\eqref{eq:maincssop} is also a solution of Eq.~\eqref{eq:maincssopmild}.
The converse is not true since Eq.~\eqref{eq:maincssop} requires more regularity than Eq.~\eqref{eq:maincssopmild}.
However, for sufficiently regular functions $\Phi$, both problems Eq.~\eqref{eq:maincssop} and Eq.~\eqref{eq:maincssopmild} are equivalent.
As a consequence, the concept of a mild solution is more general.

In fact, though, we want to consider solutions of Eq.~\eqref{eq:maincssopmild} with different $T>0$
\emph{simultaneously} so we should write $\Phi^T$ instead of $\Phi$ in order to indicate the dependence on $T$.
Here, we encounter the technical problem that 
the domain of definition varies with $T$. 
In order to circumvent this difficulty, we construct a ``universal'' solution $\Psi$ for all $T>0$ by translation, i.e., we set
$$ \Psi(\tau):=\Phi^T(\tau-\log T) $$ for $\tau \geq 0$.
The function $\Phi^T$ satisfies Eq.~\eqref{eq:maincssopmild} if and only if $\Psi$ satisfies the translated equation
\begin{equation}
\label{eq:maincssopuni}
\Psi(\tau)=S(\tau)\mb{u}+\int_0^\tau S(\tau-\tau')\mb{N}(\Psi(\tau'))d\tau', \quad \tau \geq 0
\end{equation}
as follows by a straightforward change of variables.
Obviously, $\Psi$ is independent of $T$ which justifies the notation.
What makes this possible is nothing but the time translation invariance of the wave maps equation.
Having constructed the solution $\Psi$, the function $\Phi^T$, for different values of $T>0$, can be obtained by simply noting that
$$ \Phi^T(\tau)=\Psi(\tau+\log T) $$
for $\tau \geq -\log T$.
Consequently, it suffices to consider the solution $\Psi$.

For a small $\varepsilon>0$, we set 
$$\omega:=\max\left \{-\tfrac{1}{2},s_0 \right \}+\varepsilon$$
where $s_0$ is the spectral bound (see Definition \ref{def:sb}).
In \cite{DSA} it has been proved that $s_0<\frac{1}{2}$, however, as already mentioned in the beginning, there is no reasonable doubt that in fact $s_0<0$, i.e., that $\psi^T$ is mode stable.
\begin{center}
\fbox{From now on we \emph{assume} that $\omega<0$.}
\end{center}
Note that the estimate for the linear time evolution from Theorem \ref{thm:linear} now reads
$$ \|S(\tau)(1-P)\mb{f}\|\lesssim e^{-|\omega|\tau}\|(1-P)\mb{f}\| $$
for all $\tau \geq 0$ and all $\mb{f} \in \mc{H}$.
We define a spacetime Banach space $\mc{X}$ by setting
$$ \mc{X}:=\left \{\Phi \in C([0,\infty),\mc{H}): \sup_{\tau > 0}e^{|\omega|\tau}\|\Phi(\tau)\|<\infty \right \} $$
and 
$$ \|\Phi\|_\mc{X}:=\sup_{\tau > 0}e^{|\omega|\tau}\|\Phi(\tau)\|. $$
The vector space $\mc{X}$ equipped with $\|\cdot\|_\mc{X}$ is a Banach space and we have encoded the decay property of the linear evolution in the definition of $\mc{X}$.
Thus, our hope is to retain the decay of the linear evolution on the nonlinear level.

Now we (formally) define the mapping
\begin{align} 
\label{eq:defK}
\mb{K}(\Psi, \mb{u})(\tau)&:=S(\tau)(1-P)\mb{u}-\int_{0}^\infty e^{\tau-\tau'} P\mb{N}(\Psi(\tau')) d\tau' \\
& \quad +\int_0^\tau S(\tau-\tau')\mb{N}(\Psi(\tau'))d\tau', \quad \tau \geq 0  \nonumber
\end{align}
where $P$ is the spectral projection from Theorem \ref{thm:linear}.
We will show in a moment (Lemma \ref{lem:KtoX} below) that $\mb{K}$ is well defined as an operator from
$\mc{X} \times \mc{H}$ to $\mc{X}$.
The relevance of $\mb{K}$ for the Cauchy problem \eqref{eq:maincssop} is based on the following observation.
Suppose there exists a function $\Psi \in \mc{X}$ that satisfies $\Psi=\mb{K}(\Psi,\mb{u})$.
By comparison with Eq.~\eqref{eq:maincssopuni}, we see that $\Psi$ is a solution of Eq.~\eqref{eq:maincssopuni} with the rather strange--looking initial data
$$ \Psi(0)=(1-P)\mb{u}-
\int_{0}^\infty e^{-\tau'} P\mb{N}(\Psi(\tau')) d\tau'. $$
However, this can also be written as
$$ \Psi(0)=\mb{u}-P \left [\mb{u}+\int_{0}^\infty e^{-\tau'} \mb{N}(\Psi(\tau')) d\tau' \right ] $$
and therefore, the original initial data $\mb{u}$ have been modified by adding an element of the unstable subspace $\langle \mb{g} \rangle$ (recall that $P\mc{H}=\langle \mb{g} \rangle$, see Theorem \ref{thm:linear}).
It is imporant to note, however, that this correction depends on the solution $\Psi$ itself.
We will discuss later how to deal with this inconvenient fact, see Sec.~\ref{sec:globalarbitrary} below.
For the moment we record the following: if we can show that there exists a function $\Psi$ with $\Psi=\mb{K}(\Psi,\mb{u})$, we obtain a global mild solution of Eq.~\eqref{eq:maincssop} for initial data that satisfy a ``co--dimension one condition''.
The reader may have noticed that there are many similarities with center--stable manifolds in the context of Hamiltonian evolution equations.
In fact, the construction of the solution by adding this type of modification is 
\emph{formally} exactly the same as e.g., in \cite{schlag}.
Consequently, it is very likely that there exists a center--stable manifold approach to the problem at hand but
we do not pursue this issue here any further.

Our goal is to apply the implicit function theorem on Banach spaces.
In order to do so, we have to show that $\mb{K}$ is continuously Fr\'echet differentiable.

\begin{lemma}
 \label{lem:KtoX}
 Let 
$(\Psi, \mb{u}) \in \mc{X} \times \mc{H}$.
Then we have 
$ \mb{K}(\Psi, \mb{u}) \in \mc{X}$.
\end{lemma}

\begin{proof}
Fix $(\Psi,\mb{u}) \in \mc{X} \times \mc{H}$.
In the following, the implicit constants of the $\lesssim$ notation may depend on $\Psi$ and  $\mb{u}$.
 Note first that the integrals in \eqref{eq:defK} are well--defined as Riemann integrals over a continuous function (with values in a Banach space).
 Furthermore, according to Lemma \ref{lem:N}, we have
 \begin{equation}
 \label{eq:proofKtoX}
\|P\mb{N}(\Psi(\tau'))\|\lesssim \|\Psi(\tau')\|^2 \gamma(\|\Psi(\tau')\|) \lesssim 
 e^{-2|\omega|\tau'} 
 \end{equation}
 for all $\tau' \geq 0$ by the continuity of $\gamma$ and this shows that the first integral in the definition of $\mb{K}$, Eq.~\eqref{eq:defK}, exists.
 We claim that $K(\Psi, \mb{u})\in C([0,\infty),\mc{H})$.
 Obviously, by the strong continuity of the semigroup, the map $\tau \mapsto S(\tau)\mb{u}: [0,\infty) \to \mc{H}$ is continuous.
 Thus, it remains to show continuity of the two integral terms.
 To this end choose an arbitrary $\tau_0 \geq 0$.
 Then we have
 \begin{align*}
  &\left \|\int_0^\infty e^{\tau_0-\tau'}P\mb{N}(\Psi(\tau'))d\tau'-
  \int_0^\infty e^{\tau-\tau'}P\mb{N}(\Psi(\tau'))d\tau' \right \|\\
  &\leq \left |e^{\tau_0}-e^\tau \right |\int_{0}^\infty e^{-\tau'}\|P\mb{N}(\Psi(\tau'))\|d\tau' \to 0 
 \end{align*}
as $\tau \to \tau_0$. Furthermore,
\begin{align*}
 &\left \|\int_{0}^{\tau_0} S(\tau_0-\tau')\mb{N}(\Psi(\tau'))d\tau'
 -\int_{0}^\tau S(\tau-\tau')\mb{N}(\Psi(\tau'))d\tau' \right \| \\
 &\leq \int_{0}^{\tau_0} \left \|S(\tau_0-\tau')\mb{N}(\Psi(\tau'))-S(\tau-\tau')\mb{N}(\Psi(\tau'))
 \right \|d\tau' \\
&\quad +\left |\int_{\tau_0}^\tau \|S(\tau-\tau')\mb{N}(\Psi(\tau'))\|d\tau' \right | \to 0
\end{align*}
as $\tau \to \tau_0$ by dominated convergence since
$$ \|S(\tau_0-\tau')P\mb{N}(\Psi(\tau'))-S(\tau-\tau')P\mb{N}(\Psi(\tau'))\| \to 0 $$
as $\tau \to \tau_0$ by the strong continuity of $S$. 
This proves that $\mb{K}(\Psi, \mb{u}) \in C([0,\infty),\mc{H})$.

It remains to show that $\|\mb{K}(\Psi,\mb{u})\|_\mc{X}<\infty$.
To this end note first that, for any $\mb{f}$, we have $S(\tau)P\mb{f}=e^{\tau}P\mb{f}$.
To see this, recall that $P\mc{H}=\langle \mb{g} \rangle$ (Theorem \ref{thm:linear}) and therefore, for any $\mb{f} \in \mc{H}$, there exists a $c(\mb{f})\in \mathbb{C}$ such that $P\mb{f}=c(\mb{f})\mb{g}$.
Since $\mb{g}$ is an eigenfunction of $L$ with eigenvalue $1$, we obtain
$$ S(\tau)P\mb{f}=c(\mb{f})S(\tau)\mb{g}=e^\tau c(\mb{f})\mb{g}=e^\tau P\mb{f} $$
as claimed.
Consequently, since $P$ commutes with $S$ (and also with the integrals since $P$ is bounded), we obtain
\begin{align*}
 \|P\mb{K}(\Psi,\mb{u})(\tau)\| &= \left \|
 -\int_{0}^\infty e^{\tau-\tau'}P\mb{N}(\Psi(\tau'))d\tau'+\int_{0}^\tau S(\tau-\tau')P\mb{N}(\Psi(\tau'))d\tau' \right \| \\
&\leq \int_\tau^\infty e^{\tau-\tau'}\|P\mb{N}(\Psi(\tau'))\|d\tau' 
\lesssim \int_\tau^\infty e^{\tau-(1+2|\omega|)\tau'}d\tau' \\
&\lesssim e^{-2|\omega|\tau} 
\end{align*}
for all $\tau \geq 0$ by Eq.~\eqref{eq:proofKtoX}.
Furthermore, we have
\begin{align*}
 \|(1-P)\mb{K}(\Psi, \mb{u})(\tau)\|&\leq \|S(\tau)(1-P)\mb{u}\|+\int_{0}^\tau \|S(\tau-\tau')(1-P)\mb{N}(\Psi(\tau'))\|d\tau' \\
 &\lesssim e^{-|\omega|\tau}\|\mb{u}\|+\int_{0}^\tau e^{-|\omega|(\tau-\tau')}\|\mb{N}(\Psi(\tau'))\|d\tau' \\
 &\lesssim e^{-|\omega|\tau}+\int_{0}^\tau e^{-|\omega|(\tau+\tau')}d\tau' \\
 &\lesssim e^{-|\omega|\tau}
\end{align*}
for all $\tau \geq 0$ by Theorem \ref{thm:linear}.
Adding up the two contributions we obtain
$$ \|\mb{K}(\Psi,\mb{u})(\tau)\|\leq \|P\mb{K}(\Psi,\mb{u})(\tau)\|+\|(1-P)\mb{K}(\Psi,\mb{u})(\tau)\|\lesssim e^{-|\omega| \tau} $$
and this yields
$$ \|\mb{K}(\Psi,\mb{u})\|_\mc{X}=\sup_{\tau > 0}e^{|\omega \tau|}\|\mb{K}(\Psi,\mb{u})(\tau)\|\lesssim 1 $$
which finishes the proof.
\end{proof}

We define
$\mb{K}_j: \mc{X} \times \mc{H}  \to \mc{X}$, $j=1,2$, by
\begin{align*} 
\mb{K}_1(\Psi, \mb{u})(\tau)&:=S(\tau)(1-P)\mb{u} \\
 \mb{K}_2(\Psi,\mb{u})(\tau)&:=-\int_{0}^\infty e^{\tau-\tau'}P\mb{N}(\Psi(\tau'))d\tau' 
 +\int_{0}^\tau S(\tau-\tau')\mb{N}(\Psi(\tau'))d\tau'
\end{align*}
for $\tau \geq 0$. 
By definition we have $\mb{K}=\mb{K}_1+\mb{K}_2$.

\begin{lemma}
\label{lem:K1}
 The mapping $\mb{K}_1: \mc{X} \times \mc{H} \to \mc{X}$ is continuously Fr\'echet differentiable 
 and its Fr\'echet derivative is given by
$$ [D \mb{K}_1(\Psi,\mb{u})(\Phi, \mb{v})](\tau)=S(\tau)(1-P)\mb{v} $$
for all $\tau \geq 0$.
\end{lemma}

\begin{proof}
We consider the partial Fr\'echet derivatives.
$\mb{K}_1$ does not depend on $\Psi$ so trivially, we have $D_1 \mb{K}_1 (\Psi,\mb{u})=\mb{0}$.
$\mb{K}_1$ is linear with respect to the second variable which shows that 
$$[D_2 \mb{K}_1(\Psi, \mb{u})\mb{v}](\tau)=S(\tau)(1-P)\mb{v}$$ 
for all $\mb{v} \in \mc{H}$.
Since $\|S(\tau)(1-P)\mb{v}\|\lesssim e^{-|\omega|\tau}\|(1-P)\mb{v}\|$ by Theorem \ref{thm:linear}
and $D_2 \mb{K}_1(\Psi,\mb{u})\mb{v}$ is independent of $(\Psi,\mb{u})$, we trivially see that 
$D_2 \mb{K}_1: \mc{X} \times \mc{H} \to \mc{B}(\mc{H}, \mc{X})$ is continuous.
Consequently, all partial Fr\'echet derivates of $\mb{K}_1$ exist and they are continuous. 
Thus, a standard result, see e.g., \cite{zeidler}, p.~140, Proposition 4.14, implies that $D\mb{K}_1$ exists, is continuous and given by
$$ D\mb{K}_1(\Psi,\mb{u})(\Phi,\mb{v})=D_1 \mb{K}_1(\Psi,\mb{u})\Phi+D_2 \mb{K}_1(\Psi, \mb{u})\mb{v}. $$
\end{proof}

\begin{lemma}
\label{lem:K2}
 The mapping $\mb{K}_2: \mc{X} \times \mc{H} \to \mc{X}$ is continuously Fr\'echet differentiable and its partial Fr\'echet derivative with respect to the first variable is given by
 \begin{align*} [D_1\mb{K}_2(\Psi, \mb{u})\Phi](\tau)=&-\int_0^\infty e^{\tau-\tau'}PD\mb{N}(\Psi(\tau'))\Phi(\tau')d\tau' \\
&+\int_0^\tau S(\tau-\tau')D\mb{N}(\Psi(\tau'))\Phi(\tau')d\tau' 
 \end{align*}
 for all $\tau \geq 0$.
\end{lemma}

\begin{proof}
 $\mb{K}_2(\Psi,\mb{u})$ does not depend on $\mb{u}$ so we have $D_2 \mb{K}_2(\Psi,\mb{u})=\mb{0}$.
 We set 
 \begin{align*}
 [\tilde{D}_1\mb{K}_2(\Psi, \mb{u})\Phi](\tau):=&-\int_0^\infty e^{\tau-\tau'}PD\mb{N}(\Psi(\tau'))\Phi(\tau')d\tau' \\
&+\int_0^\tau S(\tau-\tau')D\mb{N}(\Psi(\tau'))\Phi(\tau')d\tau' 
\end{align*}
and claim that $\tilde{D}_1\mb{K}_2=D_1 \mb{K}_2$.
In order to prove this, we have to show that
$$ \frac{\left \|\mb{K}_2(\Psi+\Phi,\mb{u})-\mb{K}_2(\Psi, \mb{u})-\tilde{D}_1 \mb{K}_2(\Psi, \mb{u})\Phi \right \|_{\mc{X}}}{\|\Phi\|_{\mc{X}}} \to 0 $$
as $\|\Phi\|_\mc{X}\to 0$.
We have
$$ P\mb{K}_2(\Psi,\mb{u})(\tau)=-\int_\tau^\infty e^{\tau-\tau'}P\mb{N}(\Psi(\tau'))d\tau' $$
and, analogously,
$$ P[\tilde{D}_1 \mb{K}_2(\Psi,\mb{u})\Phi](\tau)=-\int_\tau^\infty e^{\tau-\tau'}PD\mb{N}(\Psi(\tau'))\Phi(\tau')d\tau' $$
for all $\tau \geq 0$.
Consequently, we obtain
\begin{align}
\label{eq:proofK2}
 &\frac{1}{\|\Phi\|_\mc{X}}\|P\mb{K}_2(\Psi+\Phi, \mb{u})(\tau)-P\mb{K}_2(\Psi)(\tau)-P[\tilde{D}_1\mb{K}_2(\Psi,\mb{u})\Phi](\tau)\| \\
 &\leq \frac{e^\tau}{\|\Phi\|_\mc{X}} \int_\tau^\infty e^{-\tau'}
\|\mb{N}(\Psi(\tau')+\Phi(\tau'))-\mb{N}(\Psi(\tau'))-D\mb{N}(\Psi(\tau'))\Phi(\tau')\|d\tau' \nonumber
\end{align}
for all $\tau \geq 0$.
An application of the fundamental theorem of calculus yields
\begin{align*} \|\mb{N}(\mb{u}+\mb{v})-\mb{N}(\mb{u})-D\mb{N}(\mb{u})\mb{v}\|&=
 \left \|\int_0^1 \left [D\mb{N}(\mb{u}+t\mb{v})\mb{v}-D\mb{N}(\mb{u})\mb{v} \right ] dt \right \|\\
 &\leq \|\mb{v}\|^2\int_0^1 t \gamma_2(\|\mb{u}+t\mb{v}\|,\|\mb{u}\|)dt
\end{align*}
where we have used Lemma \ref{lem:N}.
Thus, we obtain
\begin{align}
\label{eq:proofK22}
&\|\mb{N}(\Psi(\tau')+\Phi(\tau'))-\mb{N}(\Psi(\tau'))-D\mb{N}(\Psi(\tau'))\Phi(\tau')\| \\
&\leq 
\|\Phi(\tau')\|^2 \int_0^1 t \gamma_2(\|\Psi(\tau')+t\Phi(\tau')\|,\|\Phi(\tau')\|)dt \nonumber \\
&\leq \|\Phi(\tau')\|^2 \sup_{t \in [0,1]}\sup_{\tau'>0} \gamma_2(\|\Psi(\tau')+t\Phi(\tau')\|,\|\Phi(\tau')\|) \nonumber \\
&\leq C_\Psi\;\|\Phi(\tau')\|^2 \nonumber 
\end{align}
for all $\tau' \geq 0$ where $C_\Psi>0$ is a constant that depends on $\Psi$.
Since we are interested in the limit $\|\Phi\|_\mc{X} \to 0$, we can assume that $\|\Phi\|_\mc{X} \leq 1$.
Consequently, we have
$$ \|\Psi(\tau')+t\Phi(\tau')\|\leq \|\Psi\|_\mc{X}+1 $$
for all $t \in [0,1]$, $\tau' \geq 0$ and thus, we may choose
$$ C_\Psi:=\sup_{0 \leq x,y \leq \|\Psi\|_\mc{X}+1}\gamma_2(x,y) $$
which is a finite number (depending on $\Psi$) since $\gamma_2$ is continuous.
As a consequence, from Eq.~\eqref{eq:proofK2} above, we obtain
\begin{align}
\label{eq:proofK2p1}
 &\frac{1}{\|\Phi\|_\mc{X}}\|P\mb{K}_2(\Psi+\Phi, \mb{u})(\tau)-P\mb{K}_2(\Psi, \mb{u})(\tau)-P[\tilde{D}_1\mb{K}_2(\Psi,\mb{u})\Phi](\tau)\| \\
 &\leq C_\Psi \;e^\tau\; \frac{\sup_{\tau'>\tau}e^{|\omega|\tau'} \|\Phi(\tau')\|}{\|\Phi\|_\mc{X}} \int_\tau^\infty e^{-\tau'-|\omega|\tau'}e^{|\omega| \tau'}\|\Phi(\tau')\|d\tau' \nonumber \\
 &\lesssim C_\Psi\; e^{-|\omega|\tau}\|\Phi\|_X \nonumber
 \end{align}
 for all $\tau \geq 0$.
 Furthermore, we have 
 $$ (1-P)\mb{K}_2(\Psi,\mb{u})(\tau)=\int_0^\tau S(\tau-\tau')(1-P)\mb{N}(\Psi(\tau'))d\tau' $$
 and thus,
 \begin{align*}
 &\frac{1}{\|\Phi\|_\mc{X}}\left \|(1-P)\left [\mb{K}_2(\Psi+\Phi, \mb{u})(\tau)-\mb{K}_2(\Psi, \mb{u})(\tau)-[\tilde{D}_1\mb{K}_2(\Psi,\mb{u})\Phi](\tau) \right ] \right \| \\
 &\leq \frac{1}{\|\Phi\|_\mc{X}}\int_0^\tau \left \|S(\tau-\tau')(1-P) \left [
 \mb{N}(\Psi(\tau')+\Phi(\tau'))-\mb{N}(\Psi(\tau'))-D\mb{N}(\Psi(\tau'))\Phi(\tau') \right ] \right \| d\tau' \\
 &\lesssim \frac{1}{\|\Phi\|_\mc{X}}\int_0^\tau e^{-|\omega|(\tau-\tau')}\left \|
 \mb{N}(\Psi(\tau')+\Phi(\tau'))-\mb{N}(\Psi(\tau'))-D\mb{N}(\Psi(\tau'))\Phi(\tau') \right \|d\tau'
\end{align*}
by Theorem \ref{thm:linear} and, via \eqref{eq:proofK22}, this implies
\begin{align}
\label{eq:proofK2p2}
  &\frac{1}{\|\Phi\|_\mc{X}}\left \|(1-P)\left [\mb{K}_2(\Psi+\Phi, \mb{u})(\tau)-\mb{K}_2(\Psi, \mb{u})(\tau)-[\tilde{D}_1\mb{K}_2(\Psi,\mb{u})\Phi](\tau) \right ] \right \|  \\
  &\lesssim C_\Psi \frac{e^{-|\omega| \tau}}{\|\Phi\|_\mc{X}}\int_0^\tau e^{|\omega| \tau'}\|\Phi(\tau')\|^2 d\tau' 
  \lesssim C_\Psi \;e^{-|\omega| \tau}\;\|\Phi\|_\mc{X}
\int_0^\tau e^{-|\omega|\tau'}d\tau' \nonumber \\
&\lesssim C_\Psi\;e^{-|\omega| \tau}\;\|\Phi\|_\mc{X} \nonumber
\end{align}
for all $\tau \geq 0$.
Putting together the two pieces \eqref{eq:proofK2p1} and \eqref{eq:proofK2p2}, we infer the claim
$\tilde{D}_1 \mb{K}_2=D_1 \mb{K}_2$.

It remains to show that $D_1 \mb{K}_2: \mc{X} \times \mc{H} \to \mc{B}(\mc{X})$ is continuous.
We have
\begin{align*}
 &\left \| P[D_1 \mb{K}_2(\Psi, \mb{u})\Phi](\tau)-P[D_1 \mb{K}_2(\tilde{\Psi},\tilde{\mb{u}})\Phi](\tau) \right \| \\
 &\lesssim \int_\tau^\infty e^{\tau-\tau'}\left \|D\mb{N}(\Psi(\tau'))\Phi(\tau')-D\mb{N}(\tilde{\Psi}(\tau'))\Phi(\tau') \right \|d\tau' \\
 &\lesssim \int_\tau^\infty e^{\tau-\tau'}\|\Psi(\tau')-\tilde{\Psi}(\tau') \| \|\Phi(\tau')\| \gamma_2(\|\Psi(\tau')\|, \|\tilde{\Psi}(\tau')\|) d\tau' \\
 &\leq C_\Psi \|\Psi-\tilde{\Psi}\|_\mc{X}\|\Phi\|_\mc{X}\int_\tau^\infty e^{\tau-\tau'-2|\omega|\tau'}d\tau' \\
 &\lesssim C_\Psi \;e^{-2|\omega|\tau}\|\Psi-\tilde{\Psi}\|_\mc{X}\|\Phi\|_\mc{X}
\end{align*}
for all $\tau \geq 0$ by Lemma \ref{lem:N} with a $\Psi$--dependent constant $C_\Psi$.
Since we are interested in the limit $\tilde{\Psi} \to \Psi$, we may assume that $\|\Psi-\tilde{\Psi}\|_\mc{X}\leq 1$ which implies
$$ \|\tilde{\Psi}\|_\mc{X} \leq \|\tilde{\Psi}-\Psi\|_\mc{X}+\|\Psi\|_\mc{X}\leq \|\Psi\|_\mc{X}+1 $$
and therefore, we may choose
$$ C_\Psi:=\sup_{0 \leq x,y \leq \|\Psi\|_\mc{X}+1}\gamma_2(x,y)< \infty $$
as before.
Similarly, 
\begin{align*}
 &\left \|(1-P)\left [D_1 \mb{K}_2(\Psi, \mb{u})\Phi](\tau)-[D_1 \mb{K}_2(\tilde{\Psi},\tilde{\mb{u}})\Phi](\tau) \right ] \right \| \\
 &\lesssim \int_0^\tau \left \|S(\tau-\tau')(1-P)\left [D\mb{N}(\Psi(\tau'))\Phi(\tau')-D\mb{N}(\tilde{\Psi}(\tau'))\Phi(\tau')\right ] \right \|d\tau' \\
 &\lesssim \int_0^\tau e^{-|\omega| (\tau-\tau')}\|\Psi(\tau')-\tilde{\Psi}(\tau')\|\|\Phi(\tau')\|
\gamma_2(\|\Psi(\tau')\|, \|\tilde{\Psi}(\tau')\|)d\tau' \\
&\leq C_\Psi \|\Psi-\tilde{\Psi}\|_\mc{X}\|\Phi\|_\mc{X}
\int_0^\tau e^{-|\omega|\tau-|\omega| \tau'}d\tau' \\
&\leq C_\Psi\; e^{-|\omega|\tau} \|\Psi-\tilde{\Psi}\|_\mc{X}\|\Phi\|_\mc{X}
\end{align*}
for all $\tau \geq 0$ by Lemma \ref{lem:N} and Theorem \ref{thm:linear}.
Putting those two estimates together we obtain
$$ \left \|D_1 \mb{K}_2(\Psi, \mb{u})\Phi-D_1 \mb{K}_2(\tilde{\Psi},\tilde{\mb{u}})\Phi \right \|_\mc{X}
\lesssim C_\Psi\|\Psi-\tilde{\Psi}\|_\mc{X}\|\Phi\|_\mc{X} $$
for all $\Phi \in \mc{X}$ and thus,
$$ \left \|D_1 \mb{K}_2(\Psi, \mb{u})-D_1 \mb{K}_2(\tilde{\Psi},\tilde{\mb{u}}) \right \|_{\mc{B}(\mc{X})}
\lesssim C_\Psi\|\Psi-\tilde{\Psi}\|_\mc{X} $$
which implies the continuity of $D_1 \mb{K}_2: \mc{X} \times \mc{H} \to \mc{B}(\mc{X})$ at 
$(\Psi,\mb{u})$ and $(\Psi, \mb{u})$ was arbitrary.
The claim now follows from \cite{zeidler}, p.~140, Proposition 4.14.
\end{proof}

Now we can show the existence of a global (with respect to the time variable $\tau$) mild solution to the wave maps problem.

\begin{theorem}[Global existence of a mild solution]
\label{thm:global1}
 Let $\delta>0$ be sufficiently small and assume that $\psi^T$ is mode stable. Then, for any $\mb{u} \in \mc{H}$ with $\|\mb{u}\|< \delta$, 
there exists a $\Psi(\cdot; \mb{u}) \in \mc{X}$ such that $\Psi(\cdot; \mb{u})=\mb{K}(\Psi(\cdot;\mb{u}),\mb{u})$.
Moreover, the solution $\Psi(\cdot; \mb{u})$ is unique in a sufficiently small neighborhood of 
$\mb{0}$ in $\mc{X}$ and the mapping $\mb{u} \mapsto \Psi(\cdot;\mb{u}): \mc{H} \to \mc{X}$ is continuously Fr\'echet differentiable.
\end{theorem}

\begin{proof}
Define $\tilde{\mb{K}}: \mc{X} \times \mc{H} \to \mc{X}$ by $\tilde{\mb{K}}(\Psi,\mb{u}):=\Psi-\mb{K}(\Psi,\mb{u})$.
 Then we have $\tilde{\mb{K}}(\mb{0},\mb{0})=\mb{0}$.
According to Lemmas \ref{lem:K1} and \ref{lem:K2}, $\tilde{\mb{K}}$ is continuously Fr\'echet differentiable.
 By Lemma \ref{lem:N}, we have $D\mb{N}(\mb{0})=\mb{0}$ and thus, by Lemma \ref{lem:K2}, we infer
 $D_1 \mb{K}(\mb{0}, \mb{0})=\mb{0}$.
 Consequently, $D_1 \tilde{\mb{K}}(\mb{0},\mb{0})=\mb{1}$, the identity on $\mc{X}$.
 In particular, $D_1 \tilde{\mb{K}}(\mb{0},\mb{0})$ is an isomorphism and the implicit function theorem (see e.g., \cite{zeidler}, p.~150, Theorem 4.B) yields the existence of a $\Psi(\cdot; \mb{u}) \in \mc{X}$ such that $\tilde{\mb{K}}(\Psi(\cdot; \mb{u}),\mb{u})=\mb{0}$ for all $\mb{u} \in \mc{H}$ in a sufficiently small neighborhood of the origin.
 Furthermore, $\mb{u} \mapsto \Psi(\cdot; \mb{u}): \mc{H} \to \mc{X}$ is a $C^1$--map in the sense of Fr\'echet.
\end{proof}

\subsection{Global existence for arbitrary small data}
\label{sec:globalarbitrary}
Theorem \ref{thm:global1} provides us with a global mild solution for the wave maps problem.
However, we are not able to specify the initial data freely.
Instead, the chosen initial data get modified along the one--dimensional subspace spanned by the gauge mode. This is necessary in order to suppress the instability introduced by the time translation symmetry of the original problem. 
This instability shows up because we have in fact fixed the blow up time.
However, perturbing the initial data of $\psi^T$ (i.e., choosing arbitrary $(f,g)$) does not, in general, preserve the blow up time of the solution. 
Therefore, one might hope to eliminate the shortcomings of Theorem \ref{thm:global1} by allowing for a variable blow up time.
This issue will be pursued in the current section.
In fact, there is one situation where the specified data are not modified, namely when they are chosen to be identically zero.
Then, the corresponding global solution from Theorem \ref{thm:global1} is the zero solution.
Loosely speaking, we are going to use the implicit function theorem to extend this to a neighborhood of zero.

Recall that we intend to solve the system \eqref{eq:main1stcss} and thus, 
the initial data we want to prescribe are of the form
$$ \Psi(0)(\rho)=\left ( \begin{array}{c}
 T\rho^2\left [g(T\rho)-\psi^T_t(0,T\rho) \right ] \\
T\rho \left [f'(T\rho)-\psi^T_r(0,T\rho) \right ] +2 \left [f(T\rho)-\psi^T(0,T\rho)
\right ] 
                         \end{array} \right ) $$ 
where $(f,g)$ are free functions.
This rather complicated looking expression is a consequence of the various variable transformations we have performed.
In what follows it is convenient to rewrite this in a different form.
First, we set
\begin{equation} 
\label{eq:defv}
\mb{v}(\rho):=\left ( \begin{array}{c}
   \rho^2\left [g(\rho)-\psi^1_t(0,\rho) \right ] \\
\rho \left [f'(\rho)-\psi^1_r(0,\rho) \right ] +2 \left [f(\rho)-\psi^1(0,\rho)
\right ] 
                         \end{array} \right ). 
\end{equation}                        
Note carefully that $\mb{v}$ does not depend on the blow up time $T$.
Thus, varying $\mb{v}$ is equivalent to varying $(f,g)$ and therefore, $\mb{v}$ are the free data of the problem.
The dependence on the blow up time $T$ is encoded in the operator $\mb{U}$, formally defined by
\begin{equation*} 
\mb{U}(\mb{v},T)(\rho):=\left ( \begin{array}{c}
   \frac{1}{T}v_1(T\rho)+T\rho^2 \left  [\psi_t^1(0,T\rho)-\psi_t^T(0,T\rho) \right ] \\
   v_2(T\rho)+T\rho \left [\psi_r^1(0,T\rho)-\psi_r^T(0,T\rho) \right ]+2 \left [\psi^1(0,T\rho)-\psi^T(0,T\rho) \right ] 
                             \end{array} \right ). 
\end{equation*}
This can be written in a simpler form by recalling that $\psi^T(t,r)=f_0(\frac{r}{T-t})$ and thus, 
\begin{eqnarray*}
 \psi^T(0,T\rho)=f_0(\rho) & \psi^1(0,T\rho)=f_0(T\rho) \\
 \psi_t^T(0,T\rho)=\frac{\rho}{T} f_0'(\rho) & \psi_t^1(0,T\rho)=T\rho f_0'(T\rho) \\
 \psi_r^T(0,T\rho)=\frac{1}{T}f_0'(\rho) & \psi_r^1(0,T\rho)=f_0'(T\rho).
\end{eqnarray*}
Consequently, we obtain
\begin{equation} 
\label{eq:defU}
\mb{U}(\mb{v},T)(\rho)=\left ( \begin{array}{c}
   \frac{1}{T} \left [ v_1(T\rho)+T^3\rho^3 f_0'(T\rho) \right ] -\rho^3 f_0'(\rho) \\
   v_2(T\rho)+T\rho f_0'(T\rho)+2 f_0(T\rho)-\rho f_0'(\rho)-2 f_0(\rho) 
                             \end{array} \right ). 
\end{equation}
With the above correspondence between $\mb{v}$ and $(f,g)$, we obviously have
$$ \Psi(0)=\mb{U}(\mb{v},T) $$
for the initial data.
The advantage of this new notation is that the dependencies of the data on $(f,g)$ (or, equivalently, $\mb{v}$) on the one hand, and $T$ on the other, are clearly separated now.
In order to obtain a well--posed initial value problem in the lightcone $\mc{C}_T$, the data $(f,g)$ have to be specified on the interval $[0,T]$.
This, however, introduces the technical problem that the data space depends on $T$.
In order to fix this, we restrict $T$ to have values in the interval $I:=(\frac{1}{2},\frac{3}{2})$.
This is no real restriction since our existence argument will be perturbative around $T=1$ anyway.
We need to find a space $\mc{Y}$ such that $\mb{U}$ has nice properties as a map from $\mc{Y} \times I$ to $\mc{H}$.
To this end we set
$$ \tilde{Y}_1:=\left \{u \in C^2[0,\tfrac{3}{2}]: u(0)=u'(0)=0 \right \}, \quad 
\tilde{Y}_2:=\left \{u \in C^2[0,\tfrac{3}{2}]: u(0)=0 \right \} $$
and define two norms $\|\cdot\|_{Y_1}$, $\|\cdot\|_{Y_2}$ on $\tilde{Y}_1$, $\tilde{Y}_2$, respectively, by setting
$$ \|u\|_{Y_1}^2:=\int_0^{3/2} |u''(\rho)|^2 d\rho, \quad 
 \|u\|_{Y_2}^2:=\int_0^{3/2} |u'(\rho)|^2 d\rho+\int_0^{3/2} |u''(\rho)|^2 \rho^2 d\rho. $$
Thanks to the boundary conditions for functions in $\tilde{Y}_j$, $j=1,2$, the mappings $\|\cdot\|_{Y_j}: \tilde{Y}_j \to [0,\infty)$ are really norms (not only seminorms). 
We denote by $Y_j$ the completion of $(\tilde{Y}_j, \|\cdot\|_{Y_j})$.
Furthermore, we set $\mc{Y}:=Y_1 \times Y_2$ with the canonical norm
$$ \|\mb{u}\|_\mc{Y}^2:=\|u_1\|_{Y_1}^2+\|u_2\|_{Y_2}^2. $$
By construction, $Y_1$, $Y_2$ and $\mc{Y}$ are Banach spaces.
For notational convenience we also set $\tilde{\mc{Y}}:=\tilde{Y}_1 \times \tilde{Y}_2$.
Then $\mc{Y}$ is the completion of $(\tilde{\mc{Y}},\|\cdot\|_{\mc{Y}})$ and thus, the notation is consistent.
Recall that, similarly, the Hilbert space $\mc{H}$ emerged as the completion of 
$$ \tilde{\mc{H}}:=\{u \in C^2[0,1]: u(0)=u'(0)=0\} \times \{u \in C^1[0,1]: u(0)=0\} $$
with respect to the norm 
$$ \|\mb{u}\|^2:=\int_0^1 \frac{|u_1'(\rho)|^2}{\rho^2}d\rho
+\int_0^1 |u_2'(\rho)|^2d\rho, $$
see Sec.~\ref{sec:oplin} or \cite{DSA}.
Here, we also write $\mc{H}=:H_1 \times H_2$ and $\tilde{\mc{H}}=:\tilde{H}_1 \times \tilde{H}_2$ with the obvious definitions of $H_j$, $\tilde{H}_j$, $j=1,2$.
We need a few preparing technical lemmas.

\begin{lemma}
 \label{lem:Fvcont}
 For $(v_j, T) \in \tilde{Y}_j \times I$, $j=1,2$, define
$[F_j(v_j,T)](\rho):=v_j(T\rho)$.
 Then $F_j(v_j, T) \in H_j$ and we have the estimate
 $$ \|F_j(v_j, T)-F_j(\tilde{v}_j, T)\|_j \lesssim \|v_j - \tilde{v}_j\|_{Y_j} $$
 for all $v_j, \tilde{v}_j \in \tilde{Y}_j$ and all $T \in I$.
 As a consequence, $F_j$ extends to a continuous mapping from $Y_j \times I$ to $H_j$.
\end{lemma}

\begin{proof}
 Let $v_j \in \tilde{Y}_j$.
 Since $\rho \in [0,1]$ implies $T\rho \in [0,T] \subset [0,\frac{3}{2}]$ for $T \in I$, we see that 
$\rho \mapsto v_j(T\rho)$ defines a function in $C^2[0,1]$.
 To be more precise, this function is given by 
$v_j(T\cdot)|_{[0,1]} \in C^2[0,1]$, the restriction of $v_j(T\cdot)$ to the interval $[0,1]$.
Consequently, $F_j(v_j,T) \in C^2[0,1]$, $j=1,2$.
Furthermore, the boundary conditions $v_1(0)=v_1'(0)=v_2(0)=0$ imply that
$$ [F_1(v_1,T)](0)=[F_1(v_1,T)]'(0)=[F_2(v_2,T)]'(0)=0 $$
which shows that $F_j(v_j, T) \in \tilde{H}_j \subset H_j$.
By definition, we have
$$ [F_j(v_j,T)]'(\rho)=Tv_j'(T\rho) $$
and this implies
\begin{align*}
 \|F_1(v_1,T)-F_1(\tilde{v}_1,T)\|_1^2&=T^2 \int_0^1 
\frac{|v_1'(T\rho)-\tilde{v}_1'(T\rho)|^2}{\rho^2}d\rho \\
&\lesssim T^4 \int_0^1 |v_1''(T\rho)-\tilde{v}_1''(T\rho)|^2 d\rho \\
&=T^3 \int_0^T |v_1''(\rho)-\tilde{v}_1''(\rho)|^2 d\rho \\
&\lesssim \|v_1-\tilde{v}_1\|_{Y_1}
\end{align*}
for all $v_1, \tilde{v}_1$ in $\tilde{Y}_1$ and $T \in I$ by Hardy's inequality (recall that $v_1'(0)=0$).
Analogously, we have
\begin{align*}
 \|F_2(v_2,T)-F_2(\tilde{v}_2,T)\|_2^2&=T^2 \int_0^1 |v_2'(T\rho)-\tilde{v}_2'(T\rho)|^2 d\rho \\
 &\lesssim \|v_2-\tilde{v}_2\|_{Y_2}
\end{align*}
for all $v_2,\tilde{v}_2 \in \tilde{Y}_2$ and $T \in I$ which proves the claimed estimate.
Consequently, Lemma \ref{lem:context} shows that, for any $T \in I$, the mapping $F_j(\cdot,T)$ extends to a continuous function $F_j(\cdot,T): Y_j \to H_j$ \emph{and the continuity is uniform with respect to $T$}.
Thus, in order to show that $F_j: Y_j \times I \to H_j$ is continuous, it suffices to show continuity of $F(v_j, \cdot): I \to H_j$ for any fixed $v_j \in Y_j$.
To see this, note that
\begin{align}
\label{eq:proofFcont}
\|F_1(v_1,T)-F_1(v_1,\tilde{T})\|_1^2 &\lesssim \int_0^1 \left |T^2 v_1''(T\rho)-\tilde{T}^2 v_1''(\tilde{T}\rho) \right |^2 d\rho \\
&\lesssim \int_0^1 \left |T^2 v_1''(T\rho)-T^2 \tilde{v}_1''(T\rho) \right |^2 d\rho+
\int_0^1 \left |T^2 \tilde{v}_1''(T\rho)-\tilde{T}^2 \tilde{v}_1''(\tilde{T}\rho) \right |^2 d\rho \nonumber \\
&\quad +\int_0^1 \left |\tilde{T}^2 \tilde{v}_1''(\tilde{T}\rho)-\tilde{T}^2 v_1''(\tilde{T}\rho) \right |^2 d\rho \nonumber \\
&\lesssim \|v_1-\tilde{v}_1\|_{Y_1}^2
+\int_0^1 \left |T^2 \tilde{v}_1''(T\rho)-\tilde{T}^2 \tilde{v}_1''(\tilde{T}\rho) \right |^2 d\rho \nonumber
\end{align}
for all $v_1, \tilde{v}_1 \in Y_1$ and $T, \tilde{T} \in I$.
Now let $\varepsilon>0$ be arbitrary.
For any given $\delta>0$, we can find a $\tilde{v}_1 \in \tilde{Y}_1$ 
such that $\|v_1-\tilde{v}_1\|_{Y_1}<\delta$.
If $\delta>0$ is chosen small enough, \eqref{eq:proofFcont} implies
$$ \|F_1(v_1,T)-F_1(v_1,\tilde{T})\|_1\leq \tfrac{\varepsilon}{2}+C \left (\int_0^1 \left |T^2 \tilde{v}_1''(T\rho)-\tilde{T}^2 \tilde{v}_1''(\tilde{T}\rho) \right |^2 d\rho \right )^{1/2} $$
for an absolute constant $C>0$ 
and the integral goes to zero as $\tilde{T} \to T$ since $\tilde{v}_1'' \in C[0,\frac{3}{2}]$.
This shows that $F_1(v_1,\cdot): I \to H_1$ is continuous for any $v_1 \in Y_1$.
The proof for $F_2$ is completely analogous.
\end{proof}

From now on we may assume that $F_j$ is defined on all of $Y_j \times I$.
In the next lemma we show that $F_j$ has a continuous partial Fr\'echet derivative with respect to $T$.

\begin{lemma}
 \label{lem:Fvdiff}
 The mapping $F_j: Y_j \times I \to H_j$, $j=1,2$, from Lemma \ref{lem:Fvcont} is partially Fr\'echet differentiable with respect to the second variable.
 Moreover, $v_j \in Y_j$ implies that $\rho \mapsto \rho v_j'(T\rho) \in C[0,1]$ and
 the derivative 
$D_2 F_j(v_j,T): \mathbb{R} \to H_j$ at $(v_j,T)$ applied to $\lambda \in \mathbb{R}$ is given by
  $$ [D_2 F_j(v_j,T)\lambda](\rho)=\lambda \rho v_j'(T\rho). $$
Finally, $D_2 F_j: Y_j \times I \to \mc{B}(\mathbb{R},H_j)$ is continuous.
\end{lemma}

\begin{proof}
We proceed in three steps: First, we show that $F_j$ is partially Fr\'echet differentiable at any $(v_j,T) \in \tilde{Y}_j \times I$. Second, we prove that the derivative $D_2 F_j$ can be extended to a continuous function from $Y_j \times I$ to $\mc{B}(\mathbb{R},H_j)$. Finally, we show that the extended function is
the partial Fr\'echet derivative of $F_j$.

For $(v_j,T) \in \tilde{Y}_j \times I$ 
we set $[\tilde{D}_2F_j(v_j, T)\lambda](\rho):=\lambda \rho v_j'(T\rho)$ and
claim that $\tilde{D}_2 F_j(v_j,T)$ is the partial Fr\'echet derivative of $F_j$ at $(v_j,T)$.
In order to prove this, we have to show that
  $$ \frac{\|F_j(v_j, T+\lambda)-F_j(v_j, T)-\tilde{D}_2 F_j(v_j,T)\lambda\|_j}{|\lambda|} \to 0 $$
  as $|\lambda|\to 0$.
Choose $\lambda_0$ so small that $T\pm \lambda_0 \in I$.
Since we are interested in the limit $|\lambda| \to 0$, we assume in the following that $|\lambda| \in (0,\lambda_0]$.
By definition we have
$$ [F_j(v_j,T+\lambda)]'(\rho)=(T+\lambda)v_j'((T+\lambda)\rho) $$
and 
$$ [\tilde{D}_2 F_j(v_j,T)\lambda]'(\rho)=\lambda \left [T \rho v_j''(T\rho)+ v_j'(T\rho) \right ]. $$
Consequently, by the fundamental theorem of calculus, we obtain
\begin{align*}
&\left | [F_j(v_j,T+\lambda)]'(\rho)-[F_j(v_j,T)]'(\rho)-[\tilde{D}_2F_j(v_j, T)\lambda]'(\rho) \right | \\
&=\left |(T+\lambda)v_j'((T+\lambda)\rho)-Tv_j'(T\rho)-
\lambda \left [T \rho v_j''(T\rho)+ v_j'(T\rho) \right ] \right | \\
&=\left | \lambda \int_0^1 \left [(T+h\lambda)\rho v_j''((T+h\lambda)\rho)+v_j'((T+h\lambda)\rho) \right ] dh
 -\lambda \left [T \rho v_j''(T\rho)+ v_j'(T\rho) \right ] \right | \\
&\leq |\lambda| \left [ \rho \left |\int_0^1 \left [(T+h\lambda) v_j''((T+h\lambda)\rho)-T v_j''(T\rho) \right ]dh \right |
+\left |\int_0^1 \left [v_j'((T+h\lambda)\rho)-v_j'(T\rho) \right ]dh \right | \right ]
\end{align*}
and this shows
\begin{align}
\label{eq:proofFdiff1}
 &\frac{\|F_1(v_1, T+\lambda)-F_1(v_1, T)-\tilde{D}_2 F_1(v_1,T)\lambda\|_1^2}{|\lambda|^2} \\
 &\leq \int_0^1 \left |
\int_0^1 \left [(T+h\lambda) v_1''((T+h\lambda)\rho)-T v_1''(T\rho) \right ]dh \right |^2 d\rho 
\nonumber \\
 &\quad + \int_0^1 \frac{1}{\rho^2} \left |
\int_0^1 \left [v_1'((T+h\lambda)\rho)-v_1'(T\rho) \right ] dh \right |^2 d\rho \nonumber \\
&\lesssim \int_0^1 
\int_0^1 \left |(T+h\lambda) v_1''((T+h\lambda)\rho)-T v_1''(T\rho) \right |^2dh  d\rho \nonumber
\end{align}
by Cauchy--Schwarz and Hardy's inequality.
Here, the differentiation (with respect to $\rho$) under the integral sign is justified by the fact that $v_1' \in C^1[0,\frac{3}{2}]$.
We set
$$ \chi_1(v_1,w_1,\lambda,\mu):=\left (
\int_0^1 \int_0^1 \left |(T+h\lambda) v_1''((T+h\lambda)\rho)-(T+h\mu) w_1''((T+h\mu)\rho) \right |^2dh  d\rho \right )^{1/2} $$
for $v_1,w_1 \in \tilde{Y}_1$ and $\lambda,\mu \in [-\lambda_0,\lambda_0]$.
Since $v_1'' \in C[0,\frac{3}{2}]$, we have $\chi_1(v_1,v_1,\lambda,\mu)\to 0$ as $\lambda \to \mu$.
Similarly, we have
\begin{align}
\label{eq:proofFdiff2}
 &\frac{\|F_2(v_2, T+\lambda)-F_2(v_2, T)-\tilde{D}_2 F_2(v_2,T)\lambda\|_2^2}{|\lambda|^2} \\
 &\lesssim \int_0^1 \int_0^1 \left |(T+h\lambda) \rho v_2''((T+h\lambda)\rho)-T \rho v_2''(T\rho) \right |^2 dh d\rho \nonumber
 \\
 &\quad + \int_0^1 \int_0^1  \left |v_2'((T+h\lambda)\rho)-v_2'(T\rho) \right |^2 dh d\rho \to 0 \nonumber
\end{align}
as $|\lambda| \to 0$ and we also define
\begin{align*} 
\chi_2(v_2,w_2,\lambda,\mu):=&\left ( \int_0^1 \int_0^1 \left |(T+h\lambda) \rho v_2''((T+h\lambda)\rho)
-(T+h\mu) \rho w_2''((T+h\mu)\rho) \right |^2 dh d\rho \right . \\
 &+ \left . \int_0^1 \int_0^1  
\left |v_2'((T+h\lambda)\rho)-w_2'((T+h\mu)\rho) \right |^2 dh d\rho \right )^{1/2}
\end{align*}
for $v_2,w_2 \in \tilde{Y}_2$.
The fact that $\chi_j(v_j,v_j,\lambda,0) \to 0$ as $|\lambda| \to 0$ shows that 
$F_j$ is partially Fr\'echet differentiable with respect to the second variable at $(v_j,T) \in \tilde{Y}_j \times I$ and we have 
$D_2 F_j(v_j,T)=\tilde{D}_2F_j(v_j,T)$ as claimed.

Now we turn to the second step, the continuity of the derivative. By definition, 
for all $v_j, \tilde{v}_j \in \tilde{Y}_j$, $T, \tilde{T} \in I$ and $\lambda \in \mathbb{R}$, we have
\begin{align*}
 &\|D_2 F_1(v_1,T)\lambda-D_2 F_1(\tilde{v}_1,\tilde{T})\lambda\|_1^2 \\
 &\lesssim |\lambda|^2 \left [ \int_0^1 \left |Tv_1''(T\rho)
-\tilde{T}\tilde{v}_1''(\tilde{T}\rho) \right |^2 d\rho+
 \int_0^1 \frac{1}{\rho^2}\left | v_1'(T\rho)-\tilde{v}_1'(\tilde{T}\rho) \right |^2 d\rho \right ] \\
 &\lesssim |\lambda|^2 \int_0^1 \left |T v_1''(T\rho)-\tilde{T} \tilde{v}_1''(\tilde{T}\rho) \right |^2 d\rho 
 \end{align*}
 and therefore,
 $\|D_2 F_1(v_1,T)\lambda-D_2 F_1(\tilde{v}_1,T)\lambda\|_1^2 
\lesssim |\lambda|^2 \|v_1-\tilde{v}_1\|_{Y_1}^2$.
Analogously,
\begin{align*}
 &\|D_2 F_2(v_2,T)\lambda-D_2 F_2(\tilde{v}_2,\tilde{T})\lambda\|_2^2 \\
  &\lesssim |\lambda|^2 \left [ \int_0^1 \left |Tv_2''(T\rho)-\tilde{T}\tilde{v}_2''(\tilde{T}\rho) \right |^2 \rho^2 d\rho+
 \int_0^1 \left | v_2'(T\rho)-\tilde{v}_2'(\tilde{T}\rho) \right |^2 d\rho \right ],
 \end{align*}
 and thus, $\|D_2 F_2(v_2,T)\lambda-D_2 F_2(\tilde{v}_2,\tilde{T})\lambda\|_2^2
\lesssim |\lambda|^2 \|v_2-\tilde{v}_2\|_{Y_2}^2$.
Consequently, we obtain
$$ \|D_2 F_j(v_j,T)-D_2 F_j(\tilde{v}_j,T)\|_{\mc{B}(\mathbb{R},H_j)}\lesssim \|v_j-\tilde{v}_j\|_{Y_j} $$
for all $v_j, \tilde{v}_j \in \tilde{Y}_j$ and $T \in I$.
This estimate implies that, for any fixed $T \in I$, $D_2 F_j(\cdot,T)$ 
can be extended to a continuous map $D_2 F_j(\cdot,T): Y_j \to \mc{B}(\mathbb{R},H_j)$ (use Lemma \ref{lem:context} and recall that $\mc{B}(\mathbb{R},H_j)$ is a Banach space)
\emph{and the continuity is uniform with respect to $T \in I$}.
Now let $v_1 \in Y_1$ and $\varepsilon>0$ be arbitrary and choose a $\tilde{v}_1 \in \tilde{Y}_1$ with $\|v_1-\tilde{v}_1\|_{Y_1}\leq \varepsilon$.
Then we have
\begin{align*}
 &\|D_2 F_1(v_1,T)-D_2 F_1(v_1,\tilde{T})\|_{\mc{B}(\mathbb{R},H_1)}^2 \lesssim \int_0^1 \left |T v_1''(T\rho)-\tilde{T} v_1''(\tilde{T}\rho) \right |^2 d\rho \\
 &\lesssim \int_0^1 \left |T v_1''(T\rho)-T \tilde{v}_1''(T\rho) \right |^2 d\rho
 +\int_0^1 \left |T \tilde{v}_1''(T\rho)-\tilde{T} \tilde{v}_1''(\tilde{T}\rho) \right |^2 d\rho \\
 &\quad +\int_0^1 \left |\tilde{T} \tilde{v}_1''(\tilde{T}\rho)
-\tilde{T} v_1''(\tilde{T}\rho) \right |^2 d\rho \\
&\lesssim \|v_1-\tilde{v}_1\|_{Y_1}^2
+\int_0^1 \left |T \tilde{v}_1''(T\rho)-\tilde{T} \tilde{v}_1''(\tilde{T}\rho) \right |^2 d\rho \lesssim \varepsilon
\end{align*}
provided that $|T-\tilde{T}|$ is sufficiently small.
Here we use that $\tilde{v}_1'' \in C[0,\frac{3}{2}]$.
An analogous estimate holds for $D_2 F_2$ and we conclude that $D_2 F_j(v_j,\cdot): I \to \mc{B}(\mathbb{R},H_j)$ is continuous for any $v_j \in Y_j$.
As a consequence, we see that
$D_2 F_j: Y_j \times I \to \mc{B}(\mathbb{R},H_j)$ is continuous (recall that the continuity of $D_2 F_j(\cdot,T): Y_j \to \mc{B}(\mathbb{R},H_j)$ is uniform with respect to $T \in I$).

Finally, we want to justify the above notation, i.e., we want to show that $D_2 F_j$ is indeed the partial Fr\'echet derivative of $F_j$.
First, we prove that the estimates \eqref{eq:proofFdiff1} and \eqref{eq:proofFdiff2} are preserved by the extension.
To this end we show that, for fixed $\lambda, \mu \in [-\lambda_0,\lambda_0]$, 
$\chi_j$ can be extended to a continuous map 
$\chi_j(\cdot, \cdot, \lambda, \mu): Y_j \times Y_j  \to \mathbb{R}$.
As always, we intend to use Lemma \ref{lem:context}.
By applying Minkowski's inequality, we obtain
\begin{align*}
 &|\chi_1(v_1,w_1\lambda,\mu)-\chi_1(\tilde{v}_1,\tilde{w}_1,\lambda,\mu)| \\
 &=\left | \left (\int_0^1 \int_0^1 \left |(T+h\lambda) v_1''((T+h\lambda)\rho)
-(T+h\mu) w_1''((T+h\mu)\rho) \right |^2 dh d\rho \right )^{1/2} \right .\\
 &\quad - \left .\left ( \int_0^1 \int_0^1 \left |(T+h\lambda) \tilde{v}_1''((T+h\lambda)\rho)
-(T+h\mu) \tilde{w}_1''((T+h\mu)\rho) \right |^2 dh  d\rho \right )^{1/2} \right | \\
 &\leq \left (\int_0^1 \int_0^1 \left |(T+h\lambda) v_1''((T+h\lambda)\rho)
-(T+h\mu) w_1''((T+h\mu)\rho) \right . \right . \\
& \quad \left . \left .-(T+h\lambda) \tilde{v}_1''((T+h\lambda)\rho)+(T+h\mu) 
\tilde{w}_1''((T+h\mu)\rho) \right |^2 dh d\rho \right )^{1/2}
\end{align*}
and this implies
\begin{align*}
 &|\chi_1(v_1,w_1,\lambda,\mu)-\chi_1(\tilde{v}_1,\tilde{w}_1,\lambda,\mu)| \\
& \leq \left (
\int_0^1 \int_0^1 |T+h\lambda|^2 \left | v_1''((T+h\lambda)\rho)
-\tilde{v}_1''((T+h\lambda)\rho) \right |^2 dh d\rho \right )^{1/2} \\
&\quad +\left (
\int_0^1 \int_0^1 |T+h\mu|^2 \left | w_1''((T+h\mu)\rho)
-\tilde{w}_1''((T+h\mu)\rho) \right |^2 dh d\rho \right )^{1/2}
=:I_1^{1/2}+I_2^{1/2}.
\end{align*}
We apply Fubini's theorem (this is justified since we assume $v_j, \tilde{v}_j, w_j, \tilde{w}_j \in C^2[0,\frac{3}{2}]$) and a change of variables to obtain
\begin{align*}
 I_1&=\int_0^1 |T+h\lambda| \int_0^{T+h\lambda} \left | v_1''(\rho)
-\tilde{v}_1''(\rho) \right |^2 d\rho dh \lesssim \|v_1-\tilde{v}_1\|_{Y_1}^2
\end{align*}
and, analogously,
$I_2\lesssim \|w_1-\tilde{w}_1\|_{Y_1}^2$
for all $v_1, w_1, \tilde{v}_1, \tilde{w}_1 \in \tilde{Y}_1$.
Note carefully that the implicit constants in these estimates are independent of $\lambda, \mu$ (as long as $\lambda, \mu \in [-\lambda_0, \lambda_0]$ which we assume throughout). 
The function $\chi_2$ can be treated in a completely analogous way and putting everything together we arrive at
$$ |\chi_j(v_j,w_j,\lambda,\mu)-\chi_j(\tilde{v}_j,\tilde{w}_j,\lambda,\mu)|\lesssim 
\|(v_j,w_j)-(\tilde{v}_j,\tilde{w}_j)\|_{Y_j \times Y_j} $$
for all $(v_j,w_j), (\tilde{v}_j,\tilde{w}_j) \in \tilde{Y}_j \times \tilde{Y}_j$. Consequently, Lemma \ref{lem:context} shows that $\chi_j$ can be (uniquely) extended
to a continuous function $\chi_j(\cdot,\cdot,\lambda,\mu): Y_j \times Y_j \to \mathbb{R}$ and the estimates \eqref{eq:proofFdiff1}, \eqref{eq:proofFdiff2} remain valid for all $v_j \in Y_j$.
Furthermore, directly from the definition we have
\begin{align*} \chi_j(v_j,v_j,\lambda,0)&\lesssim \chi_j(v_j,\tilde{v}_j,\lambda,\lambda)
+\chi_j(\tilde{v}_j,\tilde{v}_j,\lambda,0)+\chi_j(\tilde{v}_j,v_j,0,0) \\
&\lesssim \|v_j-\tilde{v}_j\|_{Y_j}+\chi_j(\tilde{v}_j,\tilde{v}_j,\lambda,0)
\end{align*}
for all $v_j, \tilde{v}_j \in \tilde{Y}_j$ and $\lambda \in [-\lambda_0,\lambda_0]$ (cf.~the estimate for $I_1$ above). 
By continuity, this estimate extends to all $v_j, \tilde{v}_j \in Y_j$.
Now let $v_j \in Y_j$ and $\varepsilon>0$ be arbitrary.
Then, for any $\delta>0$, we can find an element $\tilde{v}_j \in \tilde{Y}_j$ such that
$\|v_j-\tilde{v}_j\|_{Y_j}< \delta$.
Consequently, we obtain
$$ \chi_j(v_j,v_j,\lambda,0)\leq \tfrac{1}{2}\varepsilon+\chi_j(\tilde{v}_j,\tilde{v}_j
,\lambda,0) < \varepsilon $$ provided that $\delta$ and $|\lambda|$ are sufficiently small since $\chi_j(\tilde{v}_j,\tilde{v}_j,\lambda,0)\to 0$ as $|\lambda|\to 0$.
This shows that $\chi_j(v_j,v_j,\lambda,0) \to 0$ as $|\lambda|\to 0$ and, in conjunction with the estimates \eqref{eq:proofFdiff1}, \eqref{eq:proofFdiff2}, this implies that the continuous extension $D_2 F_j$ of $\tilde{D}_2 F_j$ is indeed the partial Fr\'echet derivative of $F_j$ as suggested by the notation.
Finally, the continuous extension $D_2 F_j$ is explicitly given by
$$ D_2 F_j(v_{j},T)\lambda=\lim_{k \to \infty} \tilde{D}_2 F_j(v_{jk}, T)\lambda $$
for an arbitrary sequence $(v_{jk}) \subset \tilde{Y}_j$ with $\|v_j-v_{jk}\|_{Y_j} \to 0$ as $k \to \infty$ where the limit is taken in $H_j$.
Since convergence in $H_j$ implies pointwise convergence (see Lemma \ref{lem:Hcont}), the explicit expression
$$ D_2 F_j(v_j,T)\lambda=\lambda \rho v_j'(T\rho) $$
remains valid for all $v_j \in Y_j$. 
In particular, we see that $v_j \in Y_j$ implies that $\rho \mapsto \rho v_j'(T\rho)$ belongs to $C[0,1]$.
\end{proof}

With these preparations at hand, we go back to the mapping $\mb{U}$.
Note that the following result completely demystifies the role of the gauge mode.

\begin{lemma}
 \label{lem:U}
 The function $\mb{U}$ extends to a continuous mapping $\mb{U}: \mc{Y} \times I \to \mc{H}$.
 Furthermore, $\mb{U}$ is partially Fr\'echet differentiable with respect to the second variable, the derivative $D_2 \mb{U}: \mc{Y} \times I \to \mc{B}(\mathbb{R},\mc{H})$ is continuous and we have
 $$ D_2 \mb{U}(\mb{0},1)\lambda=2\lambda \mb{g} $$
 for all $\lambda \in \mathbb{R}$ where $\mb{g}$ is the gauge mode.
\end{lemma}

\begin{proof}
 By Eq.~\eqref{eq:defU}, we can write
$$ \mb{U}(\mb{v},T)=\left ( \begin{array}{c}
   \frac{1}{T}\left [ F_1(v_1, T)+F_1(p^3 f_0', T) \right ]-p^3 f_0' \\
   F_2(v_2, T)+F_2(p f_0', T)+2 F_2(f_0,T)-p f_0'-2 f_0
\end{array} \right ) $$
for all $\mb{v} \in \tilde{\mc{Y}}$ and $T \in I$ where $p(\rho):=\rho$.
Observe that $p^3 f_0' \in Y_1$ and $pf_0', f_0 \in Y_2$ since $f_0 \in C^\infty[0,\frac{3}{2}]$ and $f_0(0)=0$.
By Lemma \ref{lem:Fvcont}, $F_j$ uniquely extends to a continuous map $F_j: Y_j \times I \to H_j$, $j=1,2$, and this yields the unique continuous extension of $\mb{U}$ to $\mb{U}: \mc{Y} \times I \to \mc{H}$.
Furthermore, from Lemma \ref{lem:Fvdiff} it follows that $\mb{U}$ is partially Fr\'echet differentiable with respect to the second variable and $D_2 \mb{U}: \mc{Y} \times I \to \mc{B}(\mathbb{R},\mc{H})$ is continuous.
Finally, for any $v_j \in Y_j$, we have $F_j (v_j,1)=v_j$ as well as $[D_2 F_j(v_j,1)\lambda](\rho)=\lambda \rho v_j'(\rho)$ (Lemma \ref{lem:Fvdiff}) and this yields
$$ [D_2 \mb{U}(\mb{0},1)\lambda](\rho) =\lambda \left ( \begin{array}{c}
2\rho^3 f_0'(\rho)+\rho^4 f_0''(\rho) \\
3\rho f_0'(\rho)+\rho^2 f_0''(\rho)
                                \end{array} \right )
=\frac{2\lambda }{(1+\rho^2)^2} \left ( \begin{array}{c}
2 \rho^3 \\ \rho (3+\rho^2) \end{array} \right )=2\lambda \mb{g}(\rho)
 $$
 as claimed.
\end{proof}
Lemma \ref{lem:U} can be interpreted as follows. 
The mapping $T \mapsto \mb{U}(\mb{0},T)$ describes a curve in the initial data space $\mc{H}$ and $\mb{g}$ is (up to a factor $\frac{1}{2}$) the tangent vector to this curve at $\mb{U}(\mb{0},1)=\mb{0}$.

With these preparations at hand, we can return to the existence problem for the wave maps equation.
Recall that we want to construct a global solution of the integral equation
\begin{equation}
\label{eq:mainU}
\Psi(\tau)=S(\tau)\mb{U}(\mb{v},T)+\int_0^\tau S(\tau-\tau')\mb{N}(\Psi(\tau'))d\tau'
\end{equation}
for $\tau \geq 0$.
Theorem \ref{thm:global1} constitutes a preliminary step in this direction since it yields a solution $\Psi(\cdot; \mb{u}) \in \mc{X}$ to
\begin{align} 
\label{eq:mainUcorr}
\Psi(\tau; \mb{u})=&S(\tau)\mb{u}-e^\tau P \left [\mb{u}+\int_0^\infty e^{-\tau'}\mb{N}(\Psi(\tau'; \mb{u})d\tau' \right ] \\
&+\int_0^\tau S(\tau-\tau')\mb{N}(\Psi(\tau'; \mb{u}))d\tau', \quad \tau \geq 0  \nonumber
\end{align}
provided that $\|\mb{u}\|$ is sufficiently small, 
see the definition \eqref{eq:defK} of $\mb{K}$ and recall
that $S(\tau)P\mb{u}=e^\tau P\mb{u}$ (cf.~the proof of Lemma \ref{lem:KtoX}).
Moreover, the map $\mb{u} \mapsto \Psi(\cdot; \mb{u})$ is continuously Fr\'echet differentiable.
We call this map $\mb{E}$, i.e., for given $\mb{u} \in \mc{H}$ with $\|\mb{u}\|$ sufficiently small, we set $\mb{E}(\mb{u}):=\Psi(\cdot; \mb{u})$ where $\Psi(\cdot; \mb{u}) \in \mc{X}$ is the solution from Theorem \ref{thm:global1}.
$\mb{E}: \mc{U} \subset \mc{H} \to \mc{X}$ is well--defined on a sufficiently small open neighborhood $\mc{U}$ of $\mb{0}$ in $\mc{H}$ and, according to Theorem \ref{thm:global1}, $\mb{E}$ is continuously Fr\'echet differentiable on $\mc{U}$.
Note also that $\mb{E}(\mb{0})=\mb{0}$.
This follows from the fact that the solution constructed in Theorem \ref{thm:global1} is unique in a small neighborhood of $\mb{0}$ in $\mc{X}$ and obviously, $\mb{0} \in \mc{X}$ is a solution of Eq.~\eqref{eq:mainUcorr} if $\mb{u}=\mb{0}$ (recall that $\mb{N}(\mb{0})=\mb{0}$, see Lemma \ref{lem:N}).
Since $\mb{U}(\mb{0},1)=\mb{0}$ and $\mb{U}: \mc{Y} \times I \to \mc{H}$ is continuous (Lemma \ref{lem:U}), we see that $\mb{U}(\mb{v},T) \in \mc{U}$  provided that $(\mb{v}, T) \in \mc{V} \times J$ where $\mc{V}$ and $J$ are sufficiently small open neighborhoods of $\mb{0}$ in $\mc{Y}$ and $1$ in $I$, respectively.
Consequently, the mapping $\mb{E} \circ \mb{U}: \mc{V} \times J \subset \mc{Y} \times I \to \mc{X}$ is well--defined and, by Lemma \ref{lem:U} and the chain rule, $\mb{E} \circ \mb{U}$ is continuously partially Fr\'echet differentiable with respect to the second variable.
We define $\mb{F}: \mc{V} \times J \subset \mc{Y} \times I \to \langle \mb{g} \rangle$ by
$$ \mb{F}(\mb{v},T):=P \left [\mb{U}(\mb{v},T)+\int_0^\infty e^{-\tau'}\mb{N}(\mb{E}(\mb{U}(\mb{v},T))(\tau'))d\tau' \right ]. $$
Recall that $P \mc{H} = \langle \mb{g} \rangle$ (Theorem \ref{thm:linear}) and thus, $\mb{F}$ has 
indeed range in $\langle \mb{g} \rangle$.
Furthermore, observe that $F(\mb{0},1)=\mb{0}$.
For any $(\mb{v},T) \in \mc{V} \times J$, 
Theorem \ref{thm:global1} yields the existence of a solution to 
\begin{align*} 
\Psi(\tau)=&S(\tau)\mb{U}(\mb{v},T)-e^\tau \mb{F}(\mb{v},T) 
+\int_0^\tau S(\tau-\tau')\mb{N}(\Psi(\tau'))d\tau', \quad \tau \geq 0  \nonumber
\end{align*}
which is nothing but a reformulation of Eq.~\eqref{eq:mainUcorr} with $\mb{u}=\mb{U}(\mb{v},T)$.
Consequently, if we can show that, for any $\mb{v} \in \mc{V}$, we can find a $T \in J$ such that $F(\mb{v},T)=\mb{0}$, we obtain a solution to Eq.~\eqref{eq:mainU}.
This cries for an application of the implicit function theorem.

\begin{lemma}
 \label{lem:F0}
 Let $\mc{V} \subset \mc{Y}$ be a sufficiently small open neighborhood of $\mb{0}$.
Then, for any $\mb{v} \in \mc{V}$, there exists a $T \in J$ such that $\mb{F}(\mb{v},T)=\mb{0}$.
\end{lemma}

\begin{proof}
 Note first that the mapping $\mb{B}: \Psi \mapsto \int_0^\infty e^{-\tau'}\Psi(\tau')d\tau': \mc{X} \to \mc{H}$ is continuously Fr\'echet differentiable.
 This follows immediately from the fact that $\mb{B}$ is linear and, by
 $$
  \|\mb{B}\Psi\|=\left \|\int_0^\infty e^{-\tau'}\Psi(\tau')d\tau' \right \|\leq \int_0^\infty e^{-\tau'}\|\Psi(\tau')\|d\tau' \leq \sup_{\tau' > 0}\|\Psi(\tau')\|\leq \|\Psi\|_\mc{X},
 $$
also bounded.
Now we define $\tilde{\mb{N}}: \mc{X} \to \mc{X}$ by $\tilde{\mb{N}}(\Psi)(\tau):=\mb{N}(\Psi(\tau))$.
By definition, $\mb{F}$ can be written as
$$ \mb{F}(\mb{v},T)=P \left [\mb{U}(\mb{v},T)+\mb{B}\tilde{\mb{N}}(\mb{E}(\mb{U}(\mb{v},T))) \right ]. $$
We claim that $\tilde{\mb{N}}$ is continuously Fr\'echet differentiable.
To show this, we define a mapping $\tilde{D}\tilde{\mb{N}}: \mc{X} \to \mc{B}(\mc{X})$ by
 $$ [\tilde{D}\tilde{\mb{N}}(\Psi)\Phi](\tau):=D\mb{N}(\Psi(\tau))\Phi(\tau) $$
 for $\Psi,\Phi \in \mc{X}$ and $\tau \geq 0$.
 According to Lemma \ref{lem:N}, there exists a continuous function $\gamma: [0,\infty) \to [0,\infty)$ such that
 $$ \|[\tilde{D}\tilde{\mb{N}}(\Psi)\Phi](\tau)\|\leq \|\Psi(\tau)\|\|\Phi(\tau)\|\gamma(\|\Psi(\tau)\|) $$
 and this shows that $\tilde{D}\tilde{\mb{N}}$ is well--defined as a mapping $\mc{X} \to \mc{B}(\mc{X})$.
Invoking the fundamental theorem of calculus, we infer
\begin{align*}
 \tilde{\mb{N}}(\Psi+\Phi)(\tau)-\tilde{\mb{N}}(\Psi)(\tau)&=\mb{N}(\Psi(\tau)+\Phi(\tau))-\mb{N}(\Psi(\tau)) \\
 &=\int_0^1 D\mb{N}(\Psi(\tau)+h\Phi(\tau))\Phi(\tau)]dh \\
 &=\int_0^1 [\tilde{D}\tilde{\mb{N}}(\Psi+h\Phi)\Phi](\tau)dh
\end{align*}
and this implies
\begin{align*}
 &\|\tilde{\mb{N}}(\Psi+\Phi)(\tau)-\tilde{\mb{N}}(\Psi)(\tau)-[\tilde{D}\tilde{\mb{N}}(\Psi)\Phi](\tau)\| \\
 &\leq \int_0^1 \|[\tilde{D}\tilde{\mb{N}}(\Psi+h\Phi)\Phi](\tau)
-[\tilde{D}\tilde{\mb{N}}(\Psi)\Phi](\tau)\| dh \\
&\leq \int_0^1 h \|\Phi(\tau)\|^2 \gamma_2(\|\Psi(\tau)+h\Phi(\tau)\|, \|\Psi(\tau)\|) dh \\
&\leq C_\Psi \|\Phi\|_{\mc{X}}\|\Phi(\tau)\|
\end{align*}
for all $\Phi \in \mc{X}$ with $\|\Phi\|_\mc{X} \leq 1$ where we have used the estimate from Lemma \ref{lem:N}.
Consequently, we obtain
$$ \frac{\|\tilde{\mb{N}}(\Psi+\Phi)-\tilde{\mb{N}}(\Psi)-\tilde{D}\tilde{\mb{N}}(\Psi)\Phi\|_\mc{X}}{\|\Phi\|_\mc{X}}\leq C_\Psi \|\Phi\|_\mc{X} $$ for all $\Phi \in \mc{X}$ with $\|\Phi\|_\mc{X} \leq 1$ and this shows that $\tilde{D}\tilde{\mb{N}}$ is the Fr\'echet derivative of $\tilde{\mb{N}}$.
Another application of Lemma \ref{lem:N} yields
\begin{align*} 
\|D\tilde{\mb{N}}(\Psi)\Phi-D\tilde{\mb{N}}(\tilde{\Psi})\Phi\|_\mc{X}& \leq
\sup_{\tau>0}e^{|\omega|\tau}\left [\|\Psi(\tau)-\tilde{\Psi}(\tau)\|\|\Phi(\tau)\|\gamma_2(\|\Psi(\tau)\|,\tilde{\Psi}(\tau)\|) \right ]\\
&\leq C_\Psi \|\Psi-\tilde{\Psi}\|_\mc{X}\|\Phi\|_\mc{X}
\end{align*}
for all $\Phi \in \mc{X}$ and $\tilde{\Psi} \in \mc{X}$ with, say, $\|\Psi-\tilde{\Psi}\|\leq 1$.
This estimate shows that $D\tilde{\mb{N}}: \mc{X} \to \mc{B}(\mc{X})$ is continuous.
We conclude that $\mb{F}: \mc{V} \times J \to \langle \mb{g} \rangle$ is continuous and 
continuously Fr\'echet differentiable with respect to the second variable and by the chain rule and Lemma \ref{lem:U}, we obtain
\begin{align*} 
D_2 \mb{F}(\mb{0},1)\lambda=PD_2 \mb{U}(\mb{0},1)\lambda+\mb{B}\;\underbrace{D\tilde{\mb{N}}(\mb{E}(\mb{0}))}_{=\mb{0}}\;D\mb{E}(\mb{0})\;D_2 \mb{U}(\mb{0},1)\lambda=2\lambda \mb{g}
\end{align*}
since $\mb{U}(\mb{0},1)=\mb{0}$, $\mb{E}(\mb{0})=\mb{0}$ and $D \tilde{\mb{N}}(\mb{0})=\mb{0}$ by Lemma \ref{lem:N}.
This shows that $D_2 \mb{F}(\mb{0},1): \mathbb{R} \to \langle \mb{g} \rangle$ is an isomorphism and, since $\mb{F}(\mb{0},1)=\mb{0}$, the implicit function theorem (see e.g., \cite{zeidler}, p.~150, Theorem 4.B) yields the claim.
\end{proof}

We summarize the results of this section in a theorem.

\begin{theorem}[Global existence for the wave maps equation]
 \label{thm:global}
 Let $\varepsilon>0$ be arbitrary but small and 
assume that the fundamental self--similar solution $\psi^T$ is mode stable.
Set $\omega:=s_0+\varepsilon<0$ where $s_0$ is the spectral bound and
let $\mb{v} \in \mc{V} \subset \mc{Y}$ where $\mc{V}$ is a sufficiently small open neighborhood of $\mb{0}$ in $\mc{Y}$.
Then there exists a $T$ close to $1$ such that the equation
$$ \Phi(\tau)=S(\tau+\log T)\mb{U}(\mb{v},T)+\int_{-\log T}^\tau S(\tau-\tau')\mb{N}(\Phi(\tau'))d\tau', \quad \tau \geq -\log T $$
has a continuous solution $\Phi: [-\log T,\infty) \to \mc{H}$ satisfying 
$$ \|\Phi(\tau)\|\lesssim e^{-|\omega|\tau} $$
for all $\tau \geq -\log T$.
Consequently, $\Phi$ is a global mild solution of Eq.~\eqref{eq:maincssop} with initial data $\Phi(-\log T)=\mb{U}(\mb{v},T)$.
\end{theorem}

\subsection{Uniqueness of the solution}

Finally, we show that the solution of Theorem \ref{thm:global} is unique in the space $C([-\log T,\infty), \mc{H})$.
This is the last ingredient and it completes the proof of nonlinear stability of $\psi^T$.
The very last section \ref{sec:verylast} is nothing but the translation from the operator formulation in similarity coordinates back to the original equation \eqref{eq:main}. 
The proof of uniqueness in $C([-\log T, \infty),\mc{H})$ is necessary in order to rule out the paradoxical situation that, given initial data $(f,g)$, there exist two (or even more) solutions with the same data and some of them decay whereas others do not. 
If such a situation could occur, the whole question of stability of a solution would be meaningless.
Luckily, a standard argument can be applied to exclude this absurdity.
As before, it suffices to consider solutions of the translated problem Eq.~\eqref{eq:maincssopuni}.

\begin{lemma}
 \label{lem:unique}
 Let $\mb{u} \in \mc{H}$ and suppose $\Psi_j \in C([0,\infty),\mc{H})$, $j=1,2$, satisfies
 $$ \Psi_j(\tau)=S(\tau)\mb{u}+\int_0^\tau S(\tau-\tau')\mb{N}(\Psi_j(\tau'))d\tau' $$
 for all $\tau \geq 0$.
 Then $\Psi_1=\Psi_2$.
\end{lemma}

\begin{proof}
 Recall first that, for all $\mb{u},\mb{v} \in \mc{H}$,
 \begin{align*} 
\|\mb{N}(\mb{u})-\mb{N}(\mb{v})\|&\leq 
\int_0^1 \|D\mb{N}(\mb{v}+t(\mb{u}-\mb{v}))(\mb{u}-\mb{v})\|dt  \\
&\leq \|\mb{u}-\mb{v}\|\int_0^1 \gamma(\|\mb{v}+t(\mb{u}-\mb{v})\|)dt
 \end{align*}
 where $\gamma: [0,\infty) \to [0,\infty)$ is a suitable continuous function, see Lemma \ref{lem:N}.
Now let $\tau_0>0$ be arbitrary.
Then we have
\begin{align*}
 \|\Psi_1(\tau)-\Psi_2(\tau)\|&\leq C\int_0^\tau e^{\tau-\tau'}\|\mb{N}(\Psi_1(\tau'))-\mb{N}(\Psi_2(\tau'))\|d\tau' \\
 &\leq C(e^\tau-1)\sup_{\tau' \in [0,\tau]} \left [\|\Psi_1(\tau')-\Psi_2(\tau')\|
  \sup_{t \in [0,1]} \gamma(\|\Psi_2(\tau')+t(\Psi_1(\tau')-\Psi_2(\tau'))\|) \right ] \\
 &\leq M(\tau_0, \Psi_1, \Psi_2)(e^\tau-1)\sup_{\tau' \in [0,\tau]}\|\Psi_1(\tau')-\Psi_2(\tau')\|
\end{align*}
where 
$$ M(\tau_0, \Psi_1, \Psi_2):=C\sup_{\tau' \in [0,\tau_0]}\sup_{t \in [0,1]} \gamma(\|\Psi_2(\tau')+t(\Psi_1(\tau')-\Psi_2(\tau'))\|) $$
is finite by the continuity of $\Psi_j$ and $\gamma$.
Consequently, we can find a $\tau_1 \in (0,\tau_0]$ such that
$$ \sup_{\tau \in [0,\tau_1]}\|\Psi_1(\tau)-\Psi_2(\tau)\|\leq \tfrac{1}{2}
\sup_{\tau \in [0,\tau_1]}\|\Psi_1(\tau)-\Psi_2(\tau)\| $$
which implies $\Psi_1(\tau)=\Psi_2(\tau)$ for all $\tau \in [0,\tau_1]$.
Iterating this argument we obtain $\Psi_1(\tau)=\Psi_2(\tau)$ for all $\tau \in [0,\tau_0]$ and $\tau_0>0$ was arbitrary.
\end{proof}

\subsection{Proof of Theorem \ref{thm:main}}
\label{sec:verylast}
Set
$$ v_1(\rho):=\rho^2 [g(\rho)-\psi_t^1(0, \rho)] $$
and
$$ v_2(\rho):=\rho [f'(\rho)-\psi_r^1(0,\rho)]+2[f(\rho)-\psi^1(0,\rho)], $$
cf.~Eq.~\eqref{eq:defv}.
Then $v_1, v_2 \in C^2[0,\frac{3}{2}]$ and 
$v_1(0)=v_1'(0)=v_2(0)=0$ which shows that $v_j \in \tilde{Y}_j \subset Y_j$, $j=1,2$. 
Furthermore, we have
\begin{align*}
 \|v_1\|_{Y_1}^2&=\int_0^{3/2}|v_1''(\rho)|^2 d\rho \\
 &=\int_0^{3/2}\left |r^2 [g''(r)-\psi^1_{trr}(0,r)]+4r [g'(r)-\psi_{tr}^1(0,r)]
 +2 [ g(r)-\psi_t^1(0,r)] \right |^2 dr
\end{align*}
as well as
\begin{align*}
 \|v_2\|_{Y_2}^2&=\int_0^{3/2}\left |v_2'(\rho) \right |^2 d\rho+\int_0^{3/2} \left |v_2''(\rho) \right|^2 
 \rho^2 d\rho \\
 &=\int_0^{3/2} \left |r [f''(r)-\psi_{rr}^1(0,r)]+3[f'(r)-\psi_r^1(0,r) \right |^2 dr \\
 &\quad + \int_0^{3/2} \left |r [f'''(r)-\psi_{rrr}^1(0,r)]+4[f''(r)-\psi_{rr}^1(0,r)] \right |^2 r^2 dr
\end{align*}
and this shows
$$ \|\mb{v}\|_{\mc{Y}}=\|(f,g)-(\psi^1(0,\cdot),\psi_t^1(0,\cdot))\|_{\mc{E}'}<\delta $$
by assumption.
Consequently, if $\delta$ is small enough, we obtain $\mb{v} \in \mc{V}$ and Theorem \ref{thm:global} yields the existence of a global mild solution $\Phi \in C([-\log T,\infty),\mc{H})$ of Eq.~\eqref{eq:maincssop} with initial data $\Phi(-\log T)=\mb{U}(\mb{v},T)$ that satisfies 
\begin{equation} 
\label{eq:estsol}
\|\Phi(\tau)\|\leq C_\varepsilon e^{-|\omega|\tau}  
\end{equation}
for all $\tau \geq -\log T$ where $T \in J$, i.e., $T>0$ is close to $1$.
Furthermore, by Lemma \ref{lem:unique}, the solution $\Phi$ is unique in the class $C([-\log T,\infty),\mc{H})$.
We conclude that 
$$ (\phi_1(\tau,\rho),\phi_2(\tau,\rho))=\Phi(\tau)(\rho) $$
is a solution of the Cauchy problem Eq.~\eqref{eq:main1stcss} and therefore,
by Eq.~\eqref{eq:origpsi}, 
\begin{equation} 
\label{eq:origpsi2}
\psi(t,r)=\psi^T(t,r)+\tfrac{1}{r^2}\int_0^r r' \phi_2\left (-\log(T-t), \tfrac{r'}{T-t} \right )dr' 
\end{equation}
is a solution of the original wave maps equation \eqref{eq:maincauchy}.
Furthermore, by Eq.~\eqref{eq:origpsit}, its time derivative is given by
\begin{equation}
\label{eq:origpsit2}
\psi_t(t,r)=\psi_t^T(t,r)+\tfrac{T-t}{r^2}\phi_1\left (-\log(T-t),\tfrac{r}{T-t} \right ).
\end{equation}
As before, we write $\varphi=\psi-\psi^T$.
Now note that 
$$ r\varphi_{rr}+3\varphi_r=\partial_r\tfrac{1}{r}\partial_r(r^2\varphi) $$
and thus,
$$ \varphi_{rr}(t,r)+3\varphi_r(t,r)=\tfrac{1}{T-t}\partial_\rho \phi_2 \left (-\log(T-t),\tfrac{r}{T-t}\right ) $$
by Eq.~\eqref{eq:origpsi2}.
Similarly, we have
$$ r\varphi_{tr}+2\varphi_t=\tfrac{1}{r}\partial_r (r^2 \varphi_t) $$
and therefore, by Eq.~\eqref{eq:origpsit2}, we obtain
$$ r\varphi_{tr}(t,r)+2\varphi_t(t,r)=\tfrac{1}{r}\partial_\rho \phi_1\left (-\log(T-t), \tfrac{r}{T-t} \right ). $$
By recalling the definition of $\|\cdot\|_{\mc{E}(R)}$, this shows that
\begin{align*} \|(\varphi(t,\cdot),\varphi_t(t,\cdot))\|_{\mc{E}(T-t)}^2&=\int_0^{T-t} \left |r\varphi_{rr}(t,r)+3\varphi_r(t,r) \right|^2 dr+\int_0^{T-t} \left | r\varphi_{tr}(t,r)+2\varphi_t(t,r) \right |^2 dr \\
&=\tfrac{1}{(T-t)^2}\int_0^{T-t} \left |\partial_\rho \phi_2 \left 
(-\log(T-t),\tfrac{r}{T-t}\right ) \right |^2 dr \\
&\quad +\int_0^{T-t} \left | \frac{\partial_\rho \phi_1\left (-\log(T-t), \tfrac{r}{T-t} \right )}{r} \right |^2 dr \\
&=\tfrac{1}{T-t} \int_0^1 \left |\partial_\rho \phi_2 \left 
(-\log(T-t),\rho \right ) \right |^2 d\rho \\
&\quad +\tfrac{1}{T-t}\int_0^1 \left | \frac{\partial_\rho \phi_1\left (-\log(T-t), \rho \right )}{\rho} \right |^2 d\rho \\
&=\tfrac{1}{T-t}\|\Phi(-\log(T-t))\|^2 \\
&\leq \frac{C_\varepsilon^2}{T-t}(T-t)^{2|\omega|}
\end{align*}
by Eq.~\eqref{eq:estsol}.
This proves the claimed estimate
$$ \|(\psi(t,\cdot),\psi_t(t,\cdot))-(\psi^T(t,\cdot),\psi_t^T(t,\cdot))\|_{\mc{E}(T-t)}\leq C_\varepsilon |T-t|^{-\frac{1}{2}+|\omega|} $$
and we are done.

\bibliography{wmnlin}{}

\begin{thebibliography}{10}

\bibitem{bizon00}
Piotr Bizo{\'n}.
\newblock Equivariant self-similar wave maps from {M}inkowski spacetime into
  3-sphere.
\newblock {\em Comm. Math. Phys.}, 215(1):45--56, 2000.

\bibitem{bizon05}
Piotr Bizo{\'n}.
\newblock An unusual eigenvalue problem.
\newblock {\em Acta Phys. Polon. B}, 36(1):5--15, 2005.

\bibitem{bizon99}
Piotr Bizo{\'n}, Tadeusz Chmaj, and Zbis{\l}aw Tabor.
\newblock Dispersion and collapse of wave maps.
\newblock {\em Nonlinearity}, 13(4):1411--1423, 2000.

\bibitem{carstea}
C.~Carstea.
\newblock A construction of blow up solutions for co-rotational wave maps.
\newblock {\em Preprint arXiv:0908.1201v1}, 2009.

\bibitem{cazenave}
Thierry Cazenave, Jalal Shatah, and A.~Shadi Tahvildar-Zadeh.
\newblock Harmonic maps of the hyperbolic space and development of
  singularities in wave maps and {Y}ang-{M}ills fields.
\newblock {\em Ann. Inst. H. Poincar\'e Phys. Th\'eor.}, 68(3):315--349, 1998.

\bibitem{Chr1}
Demetrios Christodoulou and A.~Shadi Tahvildar-Zadeh.
\newblock On the asymptotic behavior of spherically symmetric wave maps.
\newblock {\em Duke Math. J.}, 71(1):31--69, 1993.

\bibitem{Chr2}
Demetrios Christodoulou and A.~Shadi Tahvildar-Zadeh.
\newblock On the regularity of spherically symmetric wave maps.
\newblock {\em Comm. Pure Appl. Math.}, 46(7):1041--1091, 1993.

\bibitem{ichdipl}
Roland Donninger.
\newblock Perturbation analysis of self--similar solutions of the {SU(2)}
  sigma-model on {M}inkowski spacetime.
\newblock {\em Diploma thesis, University of Vienna}, 2006.

\bibitem{ichpca1}
Roland Donninger and Peter~C. Aichelburg.
\newblock On the mode stability of a self-similar wave map.
\newblock {\em J. Math. Phys.}, 49(4):043515, 9, 2008.

\bibitem{ichpca2}
Roland Donninger and Peter~C. Aichelburg.
\newblock Spectral properties and linear stability of self-similar wave maps.
\newblock {\em J. Hyperbolic Differ. Equ.}, 6(2):359--370, 2009.

\bibitem{ich0}
Roland Donninger and Peter~C. Aichelburg.
\newblock A note on the eigenvalues for equivariant maps of the {SU(2)}
  sigma-model.
\newblock {\em Appl.~Math.~Comp.~Sciences}, 1(1):73--82, 2010.

\bibitem{DSA}
Roland Donninger, Birgit Sch\"orkhuber, and Peter~C. Aichelburg.
\newblock On stable self--similar blow up for equivariant wave maps: {T}he
  linearized problem.
\newblock {\em Preprint}, 2010.

\bibitem{engel}
Klaus-Jochen Engel and Rainer Nagel.
\newblock {\em One-parameter semigroups for linear evolution equations}, volume
  194 of {\em Graduate Texts in Mathematics}.
\newblock Springer-Verlag, New York, 2000.
\newblock With contributions by S. Brendle, M. Campiti, T. Hahn, G. Metafune,
  G. Nickel, D. Pallara, C. Perazzoli, A. Rhandi, S. Romanelli and R.
  Schnaubelt.

\bibitem{struwe99}
Alexandre Freire, Stefan M{\"u}ller, and Michael Struwe.
\newblock Weak convergence of wave maps from {$(1+2)$}-dimensional {M}inkowski
  space to {R}iemannian manifolds.
\newblock {\em Invent. Math.}, 130(3):589--617, 1997.

\bibitem{keel-tao}
Markus Keel and Terence Tao.
\newblock Local and global well-posedness of wave maps on {$\mathbf{R^{1+1}}$}
  for rough data.
\newblock {\em Internat. Math. Res. Notices}, (21):1117--1156, 1998.

\bibitem{klainerman-rodnianski02}
Sergiu Klainerman and Igor Rodnianski.
\newblock On the global regularity of wave maps in the critical {S}obolev norm.
\newblock {\em Internat. Math. Res. Notices}, (13):655--677, 2001.

\bibitem{kriegersurv}
J.~Krieger.
\newblock Global regularity and singularity development for wave maps.
\newblock In {\em Surveys in differential geometry. {V}ol. {XII}. {G}eometric
  flows}, volume~12 of {\em Surv. Differ. Geom.}, pages 167--201. Int. Press,
  Somerville, MA, 2008.

\bibitem{schlag}
J.~Krieger and W.~Schlag.
\newblock On the focusing critical semi-linear wave equation.
\newblock {\em Amer. J. Math.}, 129(3):843--913, 2007.

\bibitem{krieger-schlag09}
J.~Krieger and W.~Schlag.
\newblock Concentration compactness for critical wave maps.
\newblock {\em Preprint arXiv:0908.2474v1}, 2009.

\bibitem{KST08}
J.~Krieger, W.~Schlag, and D.~Tataru.
\newblock Renormalization and blow up for charge one equivariant critical wave
  maps.
\newblock {\em Invent. Math.}, 171(3):543--615, 2008.

\bibitem{krieger03}
Joachim Krieger.
\newblock Global regularity of wave maps from {${\bf R}^{3+1}$} to surfaces.
\newblock {\em Comm. Math. Phys.}, 238(1-2):333--366, 2003.

\bibitem{krieger04}
Joachim Krieger.
\newblock Global regularity of wave maps from {$\mathbf{R^{2+1}}$} to {$H^2$}.
  {S}mall energy.
\newblock {\em Comm. Math. Phys.}, 250(3):507--580, 2004.

\bibitem{nahmod02}
Andrea Nahmod.
\newblock On global existence of wave maps with critical regularity.
\newblock In {\em Surveys in differential geometry, {V}ol.\ {VIII} ({B}oston,
  {MA}, 2002)}, Surv. Differ. Geom., VIII, pages 307--335. Int. Press,
  Somerville, MA, 2003.

\bibitem{nahmod03}
Andrea Nahmod, Atanas Stefanov, and Karen Uhlenbeck.
\newblock On the well-posedness of the wave map problem in high dimensions.
\newblock {\em Comm. Anal. Geom.}, 11(1):49--83, 2003.

\bibitem{rodnianski-raphael09}
I.~Rodnianski and P.~Rapha\"el.
\newblock Stable blow up dynamics for the critical co-rotational {W}ave {M}aps
  and equivariant {Y}ang-{M}ills problems.
\newblock {\em Preprint arXiv:0911.0692v1}, 2009.

\bibitem{rodnianski-sterbenz06}
I.~Rodnianski and J.~Sterbenz.
\newblock On the {F}ormation of {S}ingularities in the {C}ritical {O}(3)
  {S}igma-{M}odel.
\newblock {\em Preprint arXiv:math/0605023v3}, 2006.

\bibitem{shatah88}
Jalal Shatah.
\newblock Weak solutions and development of singularities of the {${\rm
  SU}(2)$} {$\sigma$}-model.
\newblock {\em Comm. Pure Appl. Math.}, 41(4):459--469, 1988.

\bibitem{struwe98}
Jalal Shatah and Michael Struwe.
\newblock {\em Geometric wave equations}, volume~2 of {\em Courant Lecture
  Notes in Mathematics}.
\newblock New York University Courant Institute of Mathematical Sciences, New
  York, 1998.

\bibitem{shatah-struwe02}
Jalal Shatah and Michael Struwe.
\newblock The {C}auchy problem for wave maps.
\newblock {\em Int. Math. Res. Not.}, (11):555--571, 2002.

\bibitem{STZ94}
Jalal Shatah and A.~Shadi Tahvildar-Zadeh.
\newblock On the {C}auchy problem for equivariant wave maps.
\newblock {\em Comm. Pure Appl. Math.}, 47(5):719--754, 1994.

\bibitem{sideris}
Thomas~C. Sideris.
\newblock Global existence of harmonic maps in {M}inkowski space.
\newblock {\em Comm. Pure Appl. Math.}, 42(1):1--13, 1989.

\bibitem{struwe01}
Michael Struwe.
\newblock Uniqueness for critical nonlinear wave equations and wave maps via
  the energy inequality.
\newblock {\em Comm. Pure Appl. Math.}, 52(9):1179--1188, 1999.

\bibitem{struwe04a}
Michael Struwe.
\newblock Radially symmetric wave maps from {$(1+2)$}-dimensional {M}inkowski
  space to the sphere.
\newblock {\em Math. Z.}, 242(3):407--414, 2002.

\bibitem{struwe03}
Michael Struwe.
\newblock Equivariant wave maps in two space dimensions.
\newblock {\em Comm. Pure Appl. Math.}, 56(7):815--823, 2003.
\newblock Dedicated to the memory of J{\"u}rgen K. Moser.

\bibitem{struwe04b}
Michael Struwe.
\newblock Radially symmetric wave maps from {$(1+2)$}-dimensional {M}inkowski
  space to general targets.
\newblock {\em Calc. Var. Partial Differential Equations}, 16(4):431--437,
  2003.

\bibitem{tao09}
T.~Tao.
\newblock Global regularity of wave maps {III}-{VII}.
\newblock {\em Preprints}, 2008-2009.

\bibitem{tao01a}
Terence Tao.
\newblock Global regularity of wave maps. {I}. {S}mall critical {S}obolev norm
  in high dimension.
\newblock {\em Internat. Math. Res. Notices}, (6):299--328, 2001.

\bibitem{tao01b}
Terence Tao.
\newblock Global regularity of wave maps. {II}. {S}mall energy in two
  dimensions.
\newblock {\em Comm. Math. Phys.}, 224(2):443--544, 2001.

\bibitem{tataru-sterbenz09a}
D.~Tataru and J.~Sterbenz.
\newblock Energy dispersed large data wave maps in 2+1 dimensions.
\newblock {\em Preprint arXiv:0906.3384}, 2009.

\bibitem{tataru-sterbenz09b}
D.~Tataru and J.~Sterbenz.
\newblock Regularity of wave-maps in dimension 2+1.
\newblock {\em Preprint arXiv:0907.3148}, 2009.

\bibitem{tataru99}
Daniel Tataru.
\newblock Local and global results for wave maps. {I}.
\newblock {\em Comm. Partial Differential Equations}, 23(9-10):1781--1793,
  1998.

\bibitem{tataru01}
Daniel Tataru.
\newblock On global existence and scattering for the wave maps equation.
\newblock {\em Amer. J. Math.}, 123(1):37--77, 2001.

\bibitem{tataru05}
Daniel Tataru.
\newblock Rough solutions for the wave maps equation.
\newblock {\em Amer. J. Math.}, 127(2):293--377, 2005.

\bibitem{TS90}
Neil Turok and David Spergel.
\newblock Global texture and the microwave background.
\newblock {\em Phys. Rev. Lett.}, 64(23):2736--2739, Jun 1990.

\bibitem{zeidler}
Eberhard Zeidler.
\newblock {\em Nonlinear functional analysis and its applications. {I}}.
\newblock Springer-Verlag, New York, 1986.
\newblock Fixed-point theorems, Translated from the German by Peter R. Wadsack.

\end{thebibliography}
\bibliographystyle{plain}

\end{document}